\documentclass[final,1p,times,11.5pt]{article}

\usepackage{enumerate}
\usepackage{epsfig,alltt}
\usepackage[english]{babel}
\usepackage{blindtext}
\usepackage{dsfont}
\usepackage{amssymb}
\usepackage{amsfonts}
\usepackage{graphicx}
\usepackage{amsmath}
\usepackage{amsthm}
\usepackage{natbib}
\usepackage[utf8]{inputenc}

\usepackage[left=2.5cm,right=2.5cm,top=2.5cm,bottom=2.5cm]{geometry}

\usepackage{floatflt}
\usepackage{makeidx}
\usepackage{algorithm}

\usepackage{algorithmic}

\usepackage{url}
\usepackage{float} 
\theoremstyle{remark}

\usepackage{color}

\newcommand{\begeq}[1]{\begin{equation} \label{#1}}
\newcommand{\fineq}{\end{equation}}

\title{Bandwidth selection in deconvolution kernel distribution estimators defined by stochastic approximation method with Laplace errors}
\author{Yousri Slaoui\\
Universit\'e de Poitiers}

\begin{document}

\newtheorem{theor}{Theorem}
\newtheorem{prop}{Proposition}
\newtheorem{lemma}{Lemma}
\newtheorem{lem}{Lemma}
\newtheorem{coro}{Corollary}

\newtheorem{prof}{Proof}
\newtheorem{defi}{Definition}
\newtheorem{rem}{Remark}

\date{ }
\maketitle

\textit{Abstract}: In this paper we consider the kernel estimators of a distribution function defined by the stochastic approximation algorithm when the observation are contamined by measurement errors. It is well known that this estimators depends heavily on the choice of a smoothing parameter called the bandwidth. We propose a specific second generation plug-in method of the  deconvolution kernel distribution estimators defined by the stochastic approximation algorithm.  We show that, using the proposed bandwidth selection and the stepsize which minimize the $MISE$ (Mean Integrated Squared Error), the proposed estimator will be better than the classical one for small sample setting when the error variance is controlled by the noise to signal ratio. We corroborate these theoretical results through simulations and a real dataset.\\


\section{Introduction}

We suppose that we observe the contamined data $Y_1,\ldots,Y_n$ instead of the uncontamined data $X_1,\ldots,X_n$, where $Y_1,\ldots,Y_n$ are generated from an additive measurement error model 
\begin{eqnarray*}
Y_i=X_i+\varepsilon_i,\quad \quad i=1,\ldots,n
\end{eqnarray*}
and where $X_1,\ldots, X_n$ are independent, identically distributed random variables, and let $f_X$ and $F_X$ denote respectively the probability density and the distribution function of $X_1$, the errors $\varepsilon_1,\ldots, \varepsilon_n$ are identically distributed random variables. We assume that $X$ and $\varepsilon$ are mutually independent. The distribution function of $\varepsilon$ is denoted by $F_{\varepsilon}$, assumed known. This problem
is motivated by a wide set of practical applications in different fields such as, for example, astronomy, public health, and econometrics.
In the classical deconvolution literature, the error distributions are classified into two classes: Ordinary smooth distribution and supersmooth distribution~\citet{Fan91}. Examples of ordinary smooth distributions include
Laplacian, gamma, and symmetric gamma; examples of supersmooth distributions are normal, mixture normal and Cauchy. From a theoretical point of view, the rate of convergence cannot be faster than logarithmic for supersmooth errors,
whereas for ordinary smooth errors the rate of convergence of $F_X$ is of a much better polynomial rate. For a practical point of view, \citet{Del04} noted that the deconvolution estimators that assume Laplace error always gives better results than the Gaussian case, and as an application, they consider data from the second National Health and Nutrition Examination Survey (NHANES), which is a cohort study consisting of thousands of women who were investigated about their nutrition habits and then evaluated
for evidence of cancer. The primary variable of interest in the study of the long-term log daily saturated fat intake which was known to be imprecisely measured, for more details, see~\citet{Ste90} and~\citet{Car95}. 
Throught out this paper we suppose that $\varepsilon$ is a centred double exponentielly  distributed, also called Laplace distribution, and denoted by $\varepsilon\sim \mathcal{E}d\left(\sigma\right)$, with $\sigma$ is the scale parameter. To construct a stochastic algorithm, which approximates the function $F_X$ at a given point $x$, we define an algorithm of search of the zero of the function $h : y\to F_X(x)-y$. Following Robbins-Monro's procedure, this algorithm is defined by setting $F_{0,X}(x)\in \mathbb{R}$, and, for all $n\geq 1$, 
\begin{eqnarray*}
F_{n,X}\left(x\right)=F_{n-1,X}\left(x\right)+\gamma_nW_n,
\end{eqnarray*}
where $W_n(x)$ is an "observation" of the function $h$ at the point $F_{n-1,X}(x)$, and the stepsize $\left(\gamma_n\right)$ is a sequence of positive real numbers that goes to zero. To define $W_n(x)$, we follow the approach of \citet{Rev73,Rev77}, \citet{Tsy90}, \citet{Mok09a,Mok09b}, and \citet{Sla13,Sla14a,Sla14b} and we introduce a bandwidth $(h_n)$ (that is,
a sequence of positive real numbers that goes to zero), and a kernel $K$ (that is, a function satisfying $\int_{\mathbb{R}} K\left(x\right)dx=1$), a function $\mathcal{K}$ (that is, a function defined by $\mathcal{K}\left(z\right)=\int_{-\infty}^zK\left(u\right)du$), and a deconvoluting kernel $K^{\varepsilon}$ defined as follows:
\begin{eqnarray}\label{eq:kernel:keps}
K^{\varepsilon}\left(u\right)=\frac{1}{2\pi}\int_{\mathbb{R}}e^{-itu}\frac{\phi_K\left(t\right)}{\phi_{\varepsilon}\left(\frac{t}{h_n}\right)}dt,
\end{eqnarray}
with $\phi_L$ the Fourier transform of a function or a random variable $L$, 
and sets $W_n\left(x\right)=\mathcal{K}^{\varepsilon}\left(h_n^{-1}\left(x-Y_n\right)\right)-F_{n-1,X}\left(x\right)$.
Then, the estimator $F_{n,X}$ to estimate the distribution function $F_X$ at the point $x$ can be written as
\begin{eqnarray}\label{eq:Fn1}
F_{n,X}\left(x\right)=\left(1-\gamma_n\right)F_{n-1,X}\left(x\right)+\gamma_n\mathcal{K}^{\varepsilon}\left(h_n^{-1}\left(x-Y_n\right)\right).
\end{eqnarray}
This estimator was introduced by \cite{Sla14b} in the error-free data.\\
Now, we suppose that $F_0\left(x\right)=0$, and we let $\Pi_n=\prod_{j=1}^n\left(1-\gamma_j\right)$. Then in this paper we propose to study the following estimator of $F$ at the point $x$:
\begin{eqnarray}\label{eq:Fn}
F_{n,X}\left(x\right)
&=&\Pi_n\sum_{k=1}^n\Pi_k^{-1}\gamma_k\mathcal{K}^{\varepsilon}\left(\frac{x-Y_k}{h_k}\right).
\end{eqnarray} 
The aim of this paper is to study the properties of the proposed deconvolution kernel distribution estimator defined by the stochastic approximation algorithm~(\ref{eq:Fn1}), and its comparison with the deconvolution Nadaraya's kernel distribution estimator defined as
\begin{eqnarray}\label{eq:Nada64}
\widetilde{F}_{n,X}\left(x\right)=\frac{1}{n}\sum_{i=1}^n\mathcal{K}^{\varepsilon}\left(\frac{x-Y_i}{h_n}\right).
\end{eqnarray} 
This estimator was introduced by \cite{Nad64} in the error-free data and whose large and moderate deviation principles were established by~\cite{Sla14c} in the context of error-free data.\\
We first compute the bias and the variance of the proposed estimator $F_{n,X}$ defined by~(\ref{eq:Fn1}). It turns out that they heavily depend on the choice of the stepsize $\left(\gamma_n\right)$, and on the distribution of $\varepsilon$ and on the kernel $K$. Moreover, we proposed a plug-in estimate which minimize an estimate of the mean weighted integrated squared error, using the density function as weight function to implement the bandwith selection of the proposed estimator. 

The remainder of the paper is organized as follows. In Section~\ref{section:2}, we state our main results. Section~\ref{section:app} is devoted to our application results, first by simulations (subsection~\ref{subsection:sim}) and second using real dataset through a plug-in method (subsection~\ref{subsection:real}), we give our conclusion in Section~\ref{section:conclusion}, whereas the technical details are deferred to Section~\ref{section:proof}.
\section{Assumptions and main results} \label{section:2}
We define the following class of regularly varying sequences.
\begin{defi}
Let $\gamma \in \mathbb{R} $ and $\left(v_n\right)_{n\geq 1}$ be a nonrandom positive sequence. We say that $\left(v_n\right) \in \mathcal{GS}\left(\gamma \right)$ if
\begin{eqnarray}\label{eq:5}
\lim_{n \to +\infty} n\left[1-\frac{v_{n-1}}{v_{n}}\right]=\gamma .
\end{eqnarray}
\end{defi}
Condition~(\ref{eq:5}) was introduced by \citet{Gal73} to define regularly varying sequences (see also \citet{Boj73}, and by \citet{Mok07} in the context of stochastic approximation algorithms. Noting that the acronym $\mathcal{GS}$ stand for (Galambos \& Seneta). Typical sequences in $\mathcal{GS}\left(\gamma \right)$ are, for $b\in \mathbb{R}$, $n^{\gamma}\left(\log n\right)^{b}$, $n^{\gamma}\left(\log \log n\right)^{b}$, and so on. \\

The assumptions to which we shall refer are the following

\begin{description}
\item(A1) $\varepsilon \sim \mathcal{E}d\left(\sigma\right)$, i.e. $f_{\varepsilon}\left(x\right)=\exp\left(-\left|x\right|/\sigma\right)/\left(2\sigma\right)$.
\item(A2) The function $K$ equal to $K\left(x\right)=\left(2\pi\right)^{-1/2}\exp\left(-x^2/2\right)$.  
\item(A3) $i)$ $\left(\gamma_n\right)\in \mathcal{GS}\left(-\alpha\right)$ with $\alpha\in \left(1/2,1\right]$. \\
  $ii)$ $\left(h_n\right)\in \mathcal{GS}\left(-a\right)$ with $a\in \left(0,1\right)$.\\
  $iii)$ $\lim_{n\to \infty} \left(n\gamma_n\right)\in \left(\min\left\{2a,\left(\alpha-3a\right)/2\right\},\infty\right]$.
\item(A4) $f_X$ is bounded, differentiable, and $f_X^{\prime}$ is bounded.
\end{description}
\begin{rem} \quad \quad \quad \quad 
\begin{description}
\item Assumption $\left(A3\right)(iii)$ on the limit of $\left(n\gamma_n\right)$ as $n$ goes to infinity is usual in the framework of stochastic approximation algorithms. It implies in particular that the limit of $\left(\left[n\gamma_n\right]^{-1}\right)$ is finite. 
\end{description}
\end{rem}

Throughout this paper we shall use the following notations:

\begin{eqnarray}
\xi&=&\lim_{n\to \infty}\left(n\gamma_n\right)^{-1},\label{eq:xi}\\
\Pi_n&=&\prod_{j=1}^n\left(1-\gamma_j\right),\nonumber\\
I_1&=&\int_{\mathbb{R}}f_Y^2\left(x\right)dx,\quad I_2=\int_{\mathbb{R}}\left(f_X^{\prime}\left(x\right)\right)^2f_Y\left(x\right)dx.\nonumber
\end{eqnarray}
Our first result is the following Proposition, which gives the bias and the variance of the proposed recursive deconvolution kernel distribution function.
\begin{prop}[Bias and variance of $F_{n,X}$]\label{prop:bias:var:fn}
Let Assumptions $\left(A1\right)-\left(A4\right)$ hold, and assume that $f_X^{\prime}$ is continuous at $x$, then we have
\begin{enumerate}
\item If $a\in (0, \alpha/7]$, then
\begin{eqnarray}\label{bias:Fn}
\mathbb{E}\left[F_{n,X}\left(x\right)\right]-F_X\left(x\right)=\frac{1}{2\left(1-2a\xi\right)}h_n^2f_X^{\prime}\left(x\right)+o\left(h_n^2\right).
\end{eqnarray}
If $a\in  (\alpha/7, 1)$, then
\begin{eqnarray}\label{bias:Fn:bis}
\mathbb{E}\left[F_{n,X}\left(x\right)\right]-F_X\left(x\right)=o\left(\sqrt{\gamma_nh_n^{-3}}\right).
\end{eqnarray}
\item If $a\in [\alpha/7, 1)$, then
\begin{eqnarray}
Var\left[F_{n,X}\left(x\right)\right]&=&\frac{\sigma^4}{4\sqrt{\pi}}\frac{1}{\left(2-\left(\alpha-3a\right)\xi\right)}\frac{\gamma_n}{h_n^3}f_Y\left(x\right)+o\left(\frac{\gamma_n}{h_n^3}\right).\label{var:Fn}
\end{eqnarray}
If $a\in (0,\alpha/7)$, then
\begin{eqnarray}\label{var:Fn:bis}
Var\left[F_{n,X}\left(x\right)\right]=o\left(h_n^4\right).
\end{eqnarray}
\item If $\lim_{n\to \infty}\left(n\gamma_n\right)>\max\left\{2a, \left(\alpha-3a\right)/2\right\}$, then~(\ref{bias:Fn}) and~(\ref{var:Fn}) hold simultaneously.
\end{enumerate}
\end{prop}
The bias and the variance of the estimator $F_{n,X}$ defined by the stochastic approximation algorithm~(\ref{eq:Fn}) then heavily depend on the choice of the stepsize $\left(\gamma_n\right)$. 
Let us now state the following theorem, which gives the weak convergence rate of the estimator $F_{n,X}$ defined in~(\ref{eq:Fn}).
\begin{theor}[Weak pointwise convergence rate]\label{theo:TLC}
Let Assumptions $\left(A1\right)-\left(A4\right)$ hold, and assume that $f_X^{\prime}$ is continuous at $x$.
\begin{enumerate}
\item If there exists $c\geq 0$ such that $\gamma_n^{-1}h_n^7\to c$, then
\begin{eqnarray*}
\sqrt{\gamma_n^{-1}h_n^3}\left(F_{n,X}\left(x\right)-F_X\left(x\right) \right)
& \stackrel{\mathcal{D}}{\rightarrow} &
\mathcal{N}\left( \frac{\sqrt{c}}{2\left(1-2a\xi\right)}
f_X^{\prime}\left(x\right)
,\frac{\sigma^4}{4\sqrt{\pi}}\frac{1}{2-\left(\alpha-3a\right)\xi}f_Y\left(x\right)\right).
\end{eqnarray*}
\item If $\gamma_n^{-1}h_{n}^{7} \rightarrow \infty $, then  
\begin{eqnarray*}
\frac{1}{h_{n}^{2}}\left(F_{n,X}\left(x\right)-F_X\left(x\right) \right) \stackrel{\mathbb{P}}{\rightarrow } \frac{1}{2\left(1-2a\xi\right)}f_X^{\prime}\left(x\right),
\end{eqnarray*}
\end{enumerate}
where $\stackrel{\mathcal{D}}{\rightarrow}$ denotes the convergence in distribution, $\mathcal{N}$ the Gaussian-distribution and $\stackrel{\mathbb{P}}{\rightarrow}$ the convergence in probability.
\end{theor} 
The convergence rate of the proposed estimator~(\ref{eq:Fn}) is smaller than the ordinary kernel distribution estimator~\cite{Sla14b}. This is the price paid for not measuring $\left\{\varepsilon_i\right\}_{i=1}^n$ precisely.\\
In order to measure the quality of our proposed estimator~(\ref{eq:Fn}), we use the following quantity, 
\begin{eqnarray*}
MISE^*\left[F_{n,X}\right]&=&\mathbb{E}\int_{\mathbb{R}}\left[F_{n,X}\left(x\right)-F_X\left(x\right)\right]^2f_Y\left(x\right)dx\nonumber\\
&=&\int_{\mathbb{R}}\left(\mathbb{E}\left(F_{n,X}\left(x\right)\right)-F_X\left(x\right)\right)^2f_Y\left(x\right)dx+\int_{\mathbb{R}}Var\left(F_{n,X}\left(x\right)\right)f_Y\left(x\right)dx.
\end{eqnarray*}
Moreover, in the case $a=\alpha/7$, it follows from the proposition~\ref{prop:bias:var:fn} that 
\begin{eqnarray}\label{eq:MISE}
AMISE^*\left[F_{n,X}\right]&=&\frac{\sigma^4}{4\sqrt{\pi}\left(2-\left(\alpha-3a\right)\xi\right)}\gamma_nh_n^{-3}I_1+\frac{1}{4\left(1-2a\xi\right)^2}h_n^4I_2.
\end{eqnarray} 
Let us underline that first term in~(\ref{eq:MISE}) can be larger than the variance component of the integrated mean squared error of the proposed kernel distribution estimator with error free data~\cite{Sla14b}. Corollary~\ref{Coro:ordi} gives the $AMISE^*$ of the proposed deconvolution kernel estimators~(\ref{eq:Fn1}) using the centred double exponentialle error distribution $f_{\varepsilon}\left(x\right)=\exp\left(-\left|x\right|/\sigma\right)/\left(2\sigma\right)$.
Throughout this paper, we used the standard normal kernel. 
The following corollary gives the bandwidth which minimize the $AMISE^*$ and the corresponding $AMISE^*$.
\begin{coro}\label{Coro:ordi}
Let Assumptions $\left(A1\right)$$-$$\left(A4\right)$ hold. To minimize the $AMISE^*$ of $F_{n,X}$, the stepsize $\left(\gamma_n\right)$ must be chosen in $\mathcal{GS}\left(-1\right)$, the bandwidth $\left(h_n\right)$ must equal 
\begin{eqnarray*}
\left(\left(\frac{3\sigma^4}{4\sqrt{\pi}}\right)^{1/7}\frac{\left(1-2a\xi\right)^{2/7}}{\left(2-\left(\alpha-3a\right)\xi\right)^{1/7}}\left\{\frac{I_1}{I_2}
\right\}^{1/7}\gamma_n^{1/7}\right).
\end{eqnarray*}
Then, the asymptotic dominating term of the $MISE^*$ is
\begin{eqnarray*} AMISE^*\left[F_{n,X}\right]&=&\frac{7}{12}\left(\frac{3\sigma^4}{4\sqrt{\pi}}\right)^{4/7}\left(1-2a\xi\right)^{-6/7}\left(2-\left(\alpha-3a\right)\xi\right)^{-4/7}I_1^{4/7}I_2^{3/7}\gamma_n^{4/7}.
\end{eqnarray*}
\end{coro}
The following corollary shows that, for a special choice of the stepsize $\left(\gamma_n\right)=\left(\gamma_0n^{-1}\right)$, which fulfilled that $\lim_{n\to \infty}n\gamma_n=\gamma_0$ and that $\left(\gamma_n\right)\in \mathcal{GS}\left(-1\right)$, the optimal value for $h_n$ depend on $\gamma_0$ and then the corresponding $AMISE^*$ depend on $\gamma_0$.
\begin{coro}
Let Assumptions $\left(A1\right)$$-$$\left(A4\right)$ hold. To minimize the $AMISE^*$ of $F_{n,X}$, the stepsize $\left(\gamma_n\right)$ must be chosen in $\mathcal{GS}\left(-1\right)$, $\lim_{n\to \infty}n\gamma_n=\gamma_0$, and the bandwidth $\left(h_n\right)$ must equal 
\begin{eqnarray}\label{hoptim:recursive}
\left(\left(\frac{3\sigma^4}{8\sqrt{\pi}}\right)^{1/7}\left(\gamma_0-2/7\right)^{1/7}\left\{\frac{I_1}{I_2}
\right\}^{1/7}n^{-1/7}\right).
\end{eqnarray}
Then, the asymptotic dominating term of the $MISE^*$ is
\begin{eqnarray*}
AMISE^*\left[F_{n,X}\right]&=&\frac{7}{12}\left(\frac{3\sigma^4}{8\sqrt{\pi}}\right)^{4/7}\frac{\gamma_0^2}{\left(\gamma_0-2/7\right)^{10/7}}I_1^{4/7}I_2^{3/7}n^{-4/7}.
\end{eqnarray*}
\end{coro}
Moreover, the minimum of $\gamma_0^2\left(\gamma_0-2/7\right)^{-10/7}$ is reached at $\gamma_0=1$; then the bandwidth $\left(h_n\right)$ must equal
\begin{eqnarray}\label{hoptima:recursive}
\left(0.7634\, \sigma^{4/7}\left\{\frac{I_1}{I_2}
\right\}^{1/7}n^{-1/7}\right).
\end{eqnarray}
Then, the asymptotic dominating term of the $MISE^*$ is
\begin{eqnarray}\label{AMISE:recursive}
AMISE^*\left[F_{n,X}\right]&=&0.3883\,\sigma^{16/7}I_1^{4/7}I_2^{3/7}n^{-4/7}.
\end{eqnarray}
In order to estimate the optimal bandwidth (\ref{hoptima:recursive}), we must estimate $I_1$ and $I_2$. We followed the approach of \citet{Alt95}, which is called the plug-in estimate, and we use the following kernel estimator of $I_1$ introduced in~\cite{Sla14a} to implement the bandwidth selection in recursive kernel estimator of probability density function in the error-free context and in~\cite{Sla14b} to implement the bandwidth selection in recursive kernel estimator of distribution function also in the error-free data context:
\begin{eqnarray}\label{eq:I1hat}
\widehat{I}_1&=&\frac{\Pi_n}{n}\sum_{i,k=1}^n\Pi_k^{-1}\gamma_kb_k^{-1}K_b^{\varepsilon}\left(\frac{Y_i-Y_k}{b_k}\right),
\end{eqnarray}
where $K_b^{\varepsilon}$ is a deconvoluting kernel and $b$ is the associated bandwidth.\\
In practice, we take
\begin{eqnarray}\label{eq:h:initial}
b_n=n^{-\beta}\min\left\{\widehat{s},\frac{Q_3-Q_1}{1.349}\right\},\quad\beta \in \left(0,1\right)
\end{eqnarray} 
(see \citet{Sil86}) with $\widehat{s}$ the sample standard deviation, and $Q_1$, $Q_3$ denoting the first and third quartiles, respectively.
We followed simlar steps as in the previous works (\citet{Sla14a,Sla15a}), we prove that in order to minimize the $MISE$ of $\widehat{I}_1$, the pilot bandwidth $\left(b_n\right)$ should belong to $\mathcal{GS}\left(-2/9\right)$, and the stepsize $\left(\gamma_n\right)$ should be equal to $\left(1.93\,n^{-1}\right)$. Then to estimate $I_1$, we use $\widehat{I}_1$, with $b_n$ equal to~(\ref{eq:h:initial}), and $\beta=2/9$.\\
Furthermore, to estimate $I_2$, we followed the approach of~\cite{Sla14a} and we introduced the following kernel estimator:
\begin{eqnarray}\label{eq:I2hat}
\widehat{I}_2&=&\frac{\Pi_n^2}{n}\sum_{\substack{i,j,k=1\\j\not=k}}^n\Pi_j^{-1}\Pi_k^{-1}\gamma_j\gamma_kb_j^{\prime-2}b_k^{\prime-2}K_{b^{\prime}}^{\varepsilon \left(1\right)}\left(\frac{Y_i-Y_j}{b_j^{\prime}}\right)K_{b^{\prime}}^{\varepsilon \left(1\right)}\left(\frac{Y_i-Y_k}{b_k^{\prime}}\right),
\end{eqnarray}
where $K^{\varepsilon \left(1\right)}_{b^{\prime}}$ is the first order derivative of a deconvoluting kernel $K_{b^{\prime}}$, and $b^{\prime}$ the associated bandwidth.\\
Following similar steps as in the previous works (\citet{Sla14a,Sla15a}), we prove that in order to minimize the $MISE$ of $\widehat{I}_2$, the pilot bandwidth $\left(b_n\right)$ should belong to $\mathcal{GS}\left(-1/6\right)$, and the stepsize $\left(\gamma_n\right)$ should be equal to $\left(1.736\,n^{-1}\right)$. Then to estimate $I_2$, we use $\widehat{I}_2$, with $b_n$ equal to~(\ref{eq:h:initial}), and $\beta=1/6$.\\

Finally, the plug-in estimator of the bandwidth $\left(h_n\right)$ using the proposed algorithm~(\ref{eq:Fn}) must be equal to
\begin{eqnarray}\label{plugin:hoptima:recursive}
\left(0.7634\,\sigma^{4/7}\left\{\frac{\widehat{I}_1}{\widehat{I}_2}
\right\}^{1/7}n^{-1/7}\right).
\end{eqnarray}
Then, it follows from~(\ref{AMISE:recursive}) that the asymptotic dominating term of the $MISE^*$ can be estimated by
\begin{eqnarray*}
\widehat{AMISE^*}\left[F_{n,X}\right]&=&0.3883\,\sigma^{16/7}\widehat{I}_1^{4/7}\widehat{I}_2^{3/7}n^{-4/7}.
\end{eqnarray*}

Now, let us recall that under the assumptions $\left(A1\right)$, $\left(A2\right)$, $\left(A3\right)ii)$ and $\left(A4\right)$, the asymptotic dominating term of the $MISE^*$ of the deconvolution Nadaraya's kernel distribution estimator $\widetilde{F}_{n,X}$ is given by  
\begin{eqnarray*} AMISE^*\left[\widetilde{F}_{n,X}\right]&=&\frac{\sigma^4}{4\sqrt{\pi}}\frac{1}{nh_n^3}I_1+\frac{1}{4}h_n^4I_2.
\end{eqnarray*}
Lemma~\ref{lem:ordi} gives the $AMISE^*$ of the deconvolution Nadaraya's kernel distibution~(\ref{eq:Nada64}) estimator using the centred double exponentialle error distribution.
\begin{lemma}\label{lem:ordi}
Let Assumptions $\left(A1\right)$, $\left(A2\right)$, $\left(A3\right)ii)$ and $\left(A4\right)$ hold. To minimize the $AMISE^*$ of $\widetilde{F}_{n,X}$, the bandwidth $\left(h_n\right)$ must equal 
\begin{eqnarray}\label{hoptim:rose}
\left(0.884\sigma^{4/7}\left\{\frac{I_1}{I_2}
\right\}^{1/7}n^{-1/7}\right).
\end{eqnarray}
Then, the asymptotic dominating term of the $MISE^*$ is
\begin{eqnarray}\label{AMISE:rose}
AMISE^*\left[\widetilde{F}_{n,X}\right]&=&0.357\,\sigma^{16/7}I_1^{4/7}I_2^{3/7}n^{-4/7}.
\end{eqnarray}
\end{lemma}

To estimate the optimal bandwidth~(\ref{hoptim:rose}), we must estimate $I_1$ and $I_2$. As suggested by \citet{Hal87}, we use the following kernel estimator of $I_1$:
\begin{eqnarray}\label{eq:I1:tilde}
\widetilde{I}_1&=&\frac{1}{n\left(n-1\right)b_n}\sum_{\substack{i,j=1\\i\not=j}}^nK_b^{\varepsilon}\left(\frac{Y_i-Y_j}{b_n}\right).
\end{eqnarray}
where $\left(b_n\right)$ equal to~(\ref{eq:h:initial}),with $\beta=2/9$.
and to estimate $I_2$, we use the following kernel estimator:
\begin{eqnarray}\label{eq:I2:tilde}
\widetilde{I}_2&=&\frac{1}{n^3b_n^4}\sum_{\substack{i,j,k=1\\j\not=k}}^nK_{b^{\prime}}^{\varepsilon \left(1\right)}\left(\frac{Y_i-Y_j}{b_n^{\prime}}\right)K_{b^{\prime}}^{\varepsilon \left(1\right)}\left(\frac{Y_i-Y_k}{b_n^{\prime}}\right),
\end{eqnarray}
where $\left(b_n^{\prime}\right)$ equal to~(\ref{eq:h:initial}),with $\beta=1/6$.\\
Finally, the plug-in estimator of the bandwidth $\left(h_n\right)$ using the deconvolution Nadaraya's kernel distribution estimator~(\ref{eq:Nada64}) must be equal to
\begin{eqnarray}\label{plugin:hoptim:rose}
\left(0.884\sigma^{4/7}\left\{\frac{\widetilde{I}_1}{\widetilde{I}_2}
\right\}^{1/7}n^{-1/7}\right).
\end{eqnarray}
Then, it follows from~(\ref{AMISE:rose}) that the asymptotic dominating term of the $MISE^*$ can be estimated by
\begin{eqnarray*}
AMISE^*\left[\widetilde{F}_{n,X}\right]&=&0.357\,\sigma^{16/7}\widetilde{I}_1^{4/7}\widetilde{I}_2^{3/7}n^{-4/7}.
\end{eqnarray*}

The following Theorem gives the conditions under which the expected $AMISE^*$ of the proposed estimator $F_{n,X}$ will be smaller than the expected $AMISE^*$ of the deconvolution Nadaraya's kernel distribution estimator $\widetilde{F}_{n,X}$. Following similar steps as in~\citet{Sla14a} and~\citet{Sla15a}, we prove the following Theorem: 
\begin{theor}\label{Theo:2}
Let the assumptions $\left(A1\right)-\left(A4\right)$ hold, and the stepsize $\left(\gamma_n\right)=\left(n^{-1}\right)$. We have
\begin{eqnarray}
\frac{\mathbb{E}\left[AMISE^*\left[F_{n,X}\right]\right]}{\mathbb{E}\left[AMISE^*\left[\widetilde{F}_{n,X}\right]\right]}<1\quad \mbox{for small sample setting} 
\end{eqnarray}
Then, the expected $AMISE^*$ of the proposed estimator defined by~(\ref{eq:Fn}) is smaller than the expected $AMISE^*$ of the deconvolution Nadaraya's kernel distribution estimator defined by~(\ref{eq:Nada64}) for small sample setting.
\end{theor}

\section{Applications}\label{section:app}

The aim of our applications is to compare the performance of the deconvolution Nadaraya's kernel estimator defined in~(\ref{eq:Nada64}) with that of the proposed deconvolution distribution kernel estimators defined in~(\ref{eq:Fn1}). 
\subsection{Simulations}\label{subsection:sim}
The aim of our simulation study is to compare the performance of the deconvolution Nadaraya's kernel estimator defined in~(\ref{eq:Nada64}) with that of the proposed deconvolution distribution kernel estimators defined in~(\ref{eq:Fn}). 
\begin{description}
\item When applying $F_{n,X}$ one need to choose three quantities:
\begin{itemize}
\item The function $K$, we choose the standard normal kernel. 
\item The stepsize $\left(\gamma_n\right)=\left(\left[2/3+c\right]n^{-1}\right)$, with $c\in \left[0,1\right]$.
\item The bandwidth $\left(h_n\right)$ is chosen to be equal to~(\ref{hoptim:recursive}). To estimate $I_1$, we use the estimator $\widehat{I}_1$ given in~(\ref{eq:I1hat}), with $K_b^{\varepsilon}$ is the standard normal kernel, the pilot bandwidth $\left(b_n\right)$ is chosen to be equal to~(\ref{eq:h:initial}), with $\beta=2/9$, and $\left(\gamma_n\right)=\left(1.93\,n^{-1}\right)$. Moreover, to estimate $I_2$, we use the estimator $\widehat{I}_2$ given in~(\ref{eq:I2hat}), with $K_{b^{\prime}}^{\varepsilon}$ is the standard normal kernel, the pilot bandwidth $\left(b_n^{\prime}\right)$ is chosen to be equal to~(\ref{eq:h:initial}), with $\beta=1/6$, and $\left(\gamma_n\right)=\left(1.736\,n^{-1}\right)$.   
\end{itemize}
\item When applying $\widetilde{F}_n$ one need to choose two quantities: 
\begin{itemize}
\item The function $K$, we use the normal kernel. 
\item The bandwidth $\left(h_n\right)$ is chosen to be equal to~(\ref{hoptim:rose}). To estimate $I_1$, we used the estimator $\widetilde{I}_1$ given in~(\ref{eq:I1:tilde}), with $K_b^{\varepsilon}$ is the standard normal kernel, the pilot bandwidth $\left(b_n\right)$ is chosen to be equal to~(\ref{eq:h:initial}), with $\beta=2/9$. Moreover, to estimate $I_2$, we used the estimator $\widetilde{I}_2$ given in~(\ref{eq:I2:tilde}), with $K_{b^{\prime}}^{\varepsilon}$ is the standard normal kernel, the pilot bandwidth $\left(b_n^{\prime}\right)$ is chosen to be equal to~(\ref{eq:h:initial}), with $\beta=1/6$.
\end{itemize}
\end{description}
In order to investigate the comparison between the two estimators, we consider  $\varepsilon\sim \mathcal{E}d\left(\sigma\right)$ (i.e. centred double exponentielle with the scale parameter $\sigma$).  The error variance was controlled by the noise to signal ratio, denoted by $\texttt{NSR}$ and defined by $\texttt{NSR}=Var\left(\varepsilon\right)/Var\left(X\right)$. We consider three sample sizes: $n=25$, $n=50$ and $150$, and five distribution functions :  normal $\mathcal{N}\left(0,1/2\right)$ (see Table~\ref{Tab:1}), standard normal $\mathcal{N}\left(0,1\right)$ (see Table~\ref{Tab:2}), normal $\mathcal{N}\left(0,2\right)$  distribution (see Table~\ref{Tab:3}), the normal mixture $\frac{1}{2}\mathcal{N}\left(1/2,1\right)+\frac{1}{2}\mathcal{N}\left(-1/2,1\right)$ (see Table~\ref{Tab:4}), the exponential distribution of parameter $1/2$ $\mathcal{E}\left(1/2\right)$ (see Table~\ref{Tab:5}). For each of these five cases, $500$ samples of sizes $n=25$, $n=50$ and $150$ were generated. For each fixed $\texttt{NSR}\in \left[5\%,30\%\right]$, the number of simulations is $500$. We denote by $F_i^*$ the reference distribution, and by $F_i$ the test distribution, and then we compute the following measures : Robust Mean Relative Error ($RMRE=n^{-1}\sum_{i,\left|F_i\right|>\varepsilon}\left|\frac{F_i}{F_i^*}-1\right|$), (which simply is the mean relative error obtained by removing the observations close to zero) and the linear Correlation ($Cor=\mathbb{C}ov\left(F_i,F_i^*\right)\sigma\left(F_i\right)^{-1}\sigma\left(F_i^*\right)^{-1}$). 

\begin{table}[h]
\begin{eqnarray*}
\begin{tabular}{lccccccc}
& Nadaraya  & estimator $1$  &estimator $2$  & estimator $3$ & estimator $4$\\ \hline
$n=25$&&  &$\texttt{NSR}=5\%$  &&& \\
$RMRE$ &  $0.1109$ & $ 0.1148$ & $ 0.1089$ & $ 0.1085$ & $0.1094$\\
$Cor$  &  $0.993$ & $ 0.993$ & $0.993$ & $0.993$ & $0.993$ \\
$\texttt{CPU}$ &$13$ & $7$ & $7$ & $7$ & $7$ \\
$n=50$&&  & &&& \\
$RMRE$ &  $0.0764$ & $0.0791$ & $0.0756$ & $0.0759$ & $0.0766$\\
$Cor$  &  $0.996$ & $ 0.996$ & $ 0.996$ & $0.996$ & $0.996$ \\
$\texttt{CPU}$ &$41$ & $23$ & $24$ & $23$ & $22$ \\
$n=150$&&  &  &&&\\
$RMRE$ &  $0.0395$ & $0.0422$ & $ 0.0394$ & $ 0.0395$ & $ 0.0399$\\
$Cor$  &  $0.999$ & $ 0.999$ & $ 0.999$ & $ 0.999$ & $ 0.999$\\
$\texttt{CPU}$ &$395$ & $216$ & $212$ & $213$ & $215$ \\ 
\hline
$n=25$&&  &$\texttt{NSR}=10\%$  &&& \\
$RMRE$ &  $0.1170$ & $0.1209$ & $0.1151$ & $0.1150$ & $0.1163$\\
$Cor$  &  $0.993$ &$ 0.993$ &$ 0.993$ &$ 0.993$ &$ 0.993$ \\
$\texttt{CPU}$ &$11$ & $6$ & $6$ & $6$ & $6$ \\
$n=50$&&  &  &&& \\
$RMRE$ &  $0.0801$ & $0.0835$ & $ 0.0792$ & $0.0794$ & $0.0803$\\
$Cor$  &  $0.996$ & $0.996$ & $ 0.996$ & $ 0.996$ & $0.996$ \\
$\texttt{CPU}$ &$41$ & $24$ & $23$ & $24$ & $23$ \\
$n=150$&&  &  &&&\\
$RMRE$ &  $0.0413$ & $0.0430$ & $0.0411$ & $0.0414$ & $0.0417$\\ 
$Cor$  &  $0.999$ & $ 0.999$ & $ 0.999$ & $ 0.999$ & $ 0.999$ \\
$\texttt{CPU}$ &$394$ & $222$ & $224$ & $221$ & $218$ \\ 
\hline
$n=25$&&  &$\texttt{NSR}=20\%$  &&& \\
$RMRE$ &  $0.1225$ & $0.1269$ & $0.1207$ & $0.1203$ & $0.1215$\\
$Cor$  &  $0.992$ & $ 0.992$ & $ 0.992$ & $ 0.992$ & $0.993$ \\
$\texttt{CPU}$ &$9$ & $5$ & $5$ & $5$ & $5$ \\
$n=50$&&  &  &&& \\
$RMRE$ &  $0.0838$ & $0.0873$ & $0.0835$ & $ 0.0836$ &  $0.0842$\\
$Cor$  &  $0.996$ & $ 0.996$ & $ 0.996$ & $ 0.996$ & $ 0.996$ \\
$\texttt{CPU}$ &$39$ & $21$ & $21$ & $22$ & $23$ \\
$n=150$&&  &  &&&\\
$RMRE$ &  $0.0421$ & $0.0452$ &$0.0422$ & $0.0426$ & $0.0431$\\
$Cor$  &  $0.998$ & $ 0.998$ & $ 0.998$ & $ 0.998$ & $ 0.998$\\
$\texttt{CPU}$ &$388$ & $209$ & $207$ & $205$ & $209$ \\ 
\hline
\end{tabular}
\end{eqnarray*}
\caption{Quantitative comparison between the deconvolution Nadaraya's estimator~(\ref{eq:Nada64}) and four proposed estimators; estimator $1$ correspond to the estimator~(\ref{eq:Fn1}) with the choice of $\left(\gamma_n\right)=\left(\left[2/3\right]n^{-1}\right)$, estimator $2$ correspond to the estimator~(\ref{eq:Fn1}) with the choice of $\left(\gamma_n\right)=\left(n^{-1}\right)$, estimator $3$ correspond to the estimator~(\ref{eq:Fn1}) with the choice of $\left(\gamma_n\right)=\left(\left[4/3\right]n^{-1}\right)$ and estimator $4$ correspond to the estimator~(\ref{eq:Fn1}) with the choice of $\left(\gamma_n\right)=\left(\left[5/3\right]n^{-1}\right)$. Here we consider the normal distribution $X\sim \mathcal{N}\left(0,1/2\right)$ with $\texttt{NSR}=5\%$ in the first block, $\texttt{NSR}=10\%$ in the second block and $\texttt{NSR}=20\%$ in the last block, we consider three sample sizes $n=25$, $n=50$ and $n=150$, the number of simulations is $500$, and we compute the robust mean relative error ($RMRE$), the linear correlation ($Cor$) and the $\texttt{CPU}$ time in seconds.}\label{Tab:1}
\end{table}

\begin{figure}[h]
\begin{center}
\includegraphics[width=0.65\textwidth,angle=270,clip=true,trim=40 0 0 0]{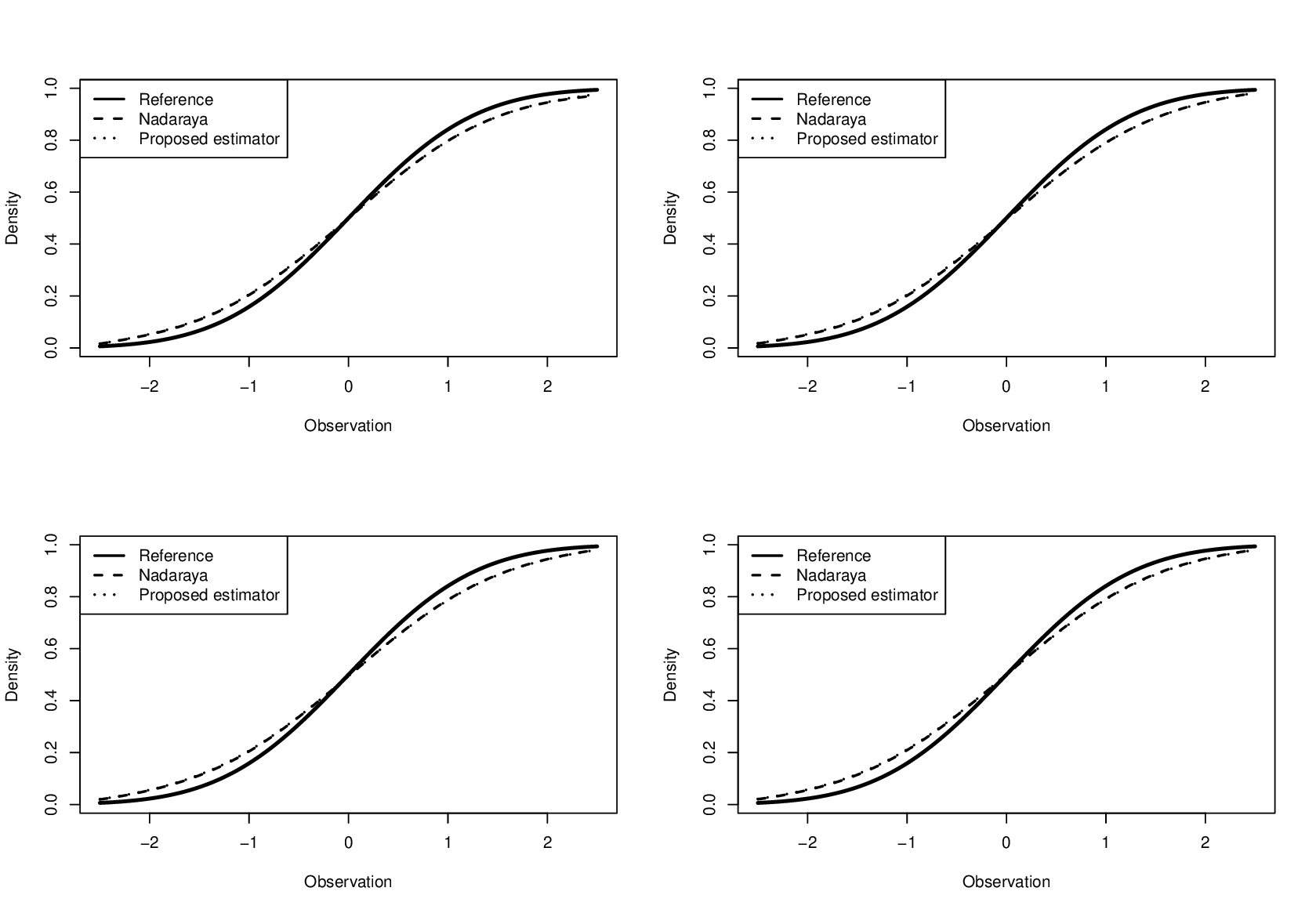}
\end{center}
\caption{Qualitative comparison between the deconvolution Nadaraya's estimator~(\ref{eq:Nada64}) and the proposed estimator~(\ref{eq:Fn1}) with the choice of the stepsize $\left(\gamma_n\right)=\left(n^{-1}\right)$, for $500$ samples of size $200$, with $\texttt{NSR}$ equal respectively to $5\%$ (in the top left panel), equal to $10\%$ (in the top right panel), equal to $20\%$ (in the down left panel) and equal to $30\%$ (in the down right panel) for the normal distribution $X\sim \mathcal{N}\left(0,1/2\right)$.}
\label{Fig:1}
\end{figure}

\begin{table}[h]
\begin{eqnarray*}
\begin{tabular}{lccccccc}
& Nadaraya  & estimator $1$  &estimator $2$  & estimator $3$ & estimator $4$\\ \hline
$n=25$&&  &$\texttt{NSR}=5\%$  &&& \\
$RMRE$ &  $0.0975$ & $ 0.1024$ & $ 0.0934$ & $ 0.0942$ & $ 0.0968$\\
$Cor$  &  $0.993$ & $ 0.993$ & $ 0.993$ & $ 0.993$ & $ 0.993$ \\
$\texttt{CPU}$ &$7$ & $4$ & $4$ & $4$ & $4$ \\
$n=50$&&  & &&& \\
$RMRE$ &  $0.0779$ & $0.0811$ & $0.0745$ & $ 0.0745$ & $0.0755$\\
$Cor$  &  $0.997$ & $ 0.997$ & $ 0.997$ & $ 0.997$ & $ 0.997$ \\
$\texttt{CPU}$ &$38$ & $20$ & $21$ & $20$ & $21$ \\
$n=150$&&  &  &&&\\
$RMRE$ &  $0.0357$ & $0.0379$ & $0.0349$ & $0.0345$ & $0.0346$\\
$Cor$  &  $0.999$ & $ 0.999$ & $ 0.999$ & $ 0.999$ & $ 0.999$\\
$\texttt{CPU}$ &$374$ & $194$ & $195$ & $193$ & $196$ \\ 
\hline
$n=25$&&  &$\texttt{NSR}=10\%$  &&& \\
$RMRE$ &  $0.1150$ & $0.1180$ & $0.1133$ & $0.1130$ & $0.1139$\\
$Cor$  &  $0.993$ & $ 0.993$ & $ 0.993$ & $ 0.993$ & $ 0.993$ \\
$\texttt{CPU}$ &$8$ & $4$ & $4$ & $4$ & $4$ \\
$n=50$&&  &  &&& \\
$RMRE$ &  $0.0797$ & $0.0805$ & $0.0772$ & $0.0773$ & $0.0778$\\
$Cor$  &  $0.997$ & $ 0.997$ & $ 0.997$ & $ 0.997$ & $0.997$ \\
$\texttt{CPU}$ &$35$ & $18$ & $17$ & $16$ & $20$ \\
$n=150$&&  &  &&&\\
$RMRE$ &  $0.0374$ & $0.0394$ & $0.0363$ & $0.0366$ & $0.0370$\\
$Cor$  &  $0.999$ & $ 0.999$ & $ 0.999$ & $0.999$ & $0.999$ \\
$\texttt{CPU}$ &$369$ & $189$ & $187$ & $186$ & $191$ \\ 
\hline
$n=25$&&  &$\texttt{NSR}=20\%$  &&& \\
$RMRE$ &  $0.1137$ & $0.1176$ & $0.1127$ & $0.1139$ & $0.1158$\\
$Cor$  &  $0.993$ & $ 0.993$ & $ 0.993$ & $ 0.993$ & $ 0.993$ \\
$\texttt{CPU}$ &$8$ & $4$ & $4$ & $4$ & $4$ \\
$n=50$&&  &  &&& \\
$RMRE$ &  $0.0834$ & $0.0864$ & $0.0832$ & $0.0842$ & $0.0851$\\
$Cor$  &  $0.996$ & $ 0.996$ & $ 0.996$ & $ 0.996$ & $ 0.996$ \\
$\texttt{CPU}$ &$37$ & $19$ & $18$ & $20$ & $21$ \\
$n=150$&&  &  &&&\\
$RMRE$ &  $0.0397$ & $0.0419$ & $0.0393$ & $0.0395$ & $0.0397$\\
$Cor$  &  $0.999$ & $ 0.999$ & $ 0.999$ & $ 0.999$ & $ 0.999$\\ 
$\texttt{CPU}$ &$379$ & $203$ & $202$ & $203$ & $205$ \\ 
\hline
\end{tabular}
\end{eqnarray*}
\caption{Quantitative comparison between the deconvolution Nadaraya's estimator~(\ref{eq:Nada64})  and four estimators; estimator $1$ correspond to the estimator~(\ref{eq:Fn1}) with the choice of $\left(\gamma_n\right)=\left(\left[2/3\right]n^{-1}\right)$, estimator $2$ correspond to the estimator~(\ref{eq:Fn1}) with the choice of $\left(\gamma_n\right)=\left(n^{-1}\right)$, estimator $3$ correspond to the estimator~(\ref{eq:Fn1}) with the choice of $\left(\gamma_n\right)=\left(\left[4/3\right]n^{-1}\right)$ and estimator $4$ correspond to the estimator~(\ref{eq:Fn1}) with the choice of $\left(\gamma_n\right)=\left(\left[5/3\right]n^{-1}\right)$. Here we consider the standard normal distribution $X\sim \mathcal{N}\left(0,1\right)$ with $\texttt{NSR}=5\%$ in the first block, $\texttt{NSR}=10\%$ in the second block and $\texttt{NSR}=20\%$ in the last block, we consider three sample sizes $n=25$, $n=50$ and $n=150$, the number of simulations is $500$, and we compute the robust mean relative error ($RMRE$), the linear correlation ($Cor$) and the $\texttt{CPU}$ time in seconds.}\label{Tab:2}
\end{table}

\begin{figure}[h]
\begin{center}
\includegraphics[width=0.65\textwidth,angle=270,clip=true,trim=40 0 0 0]{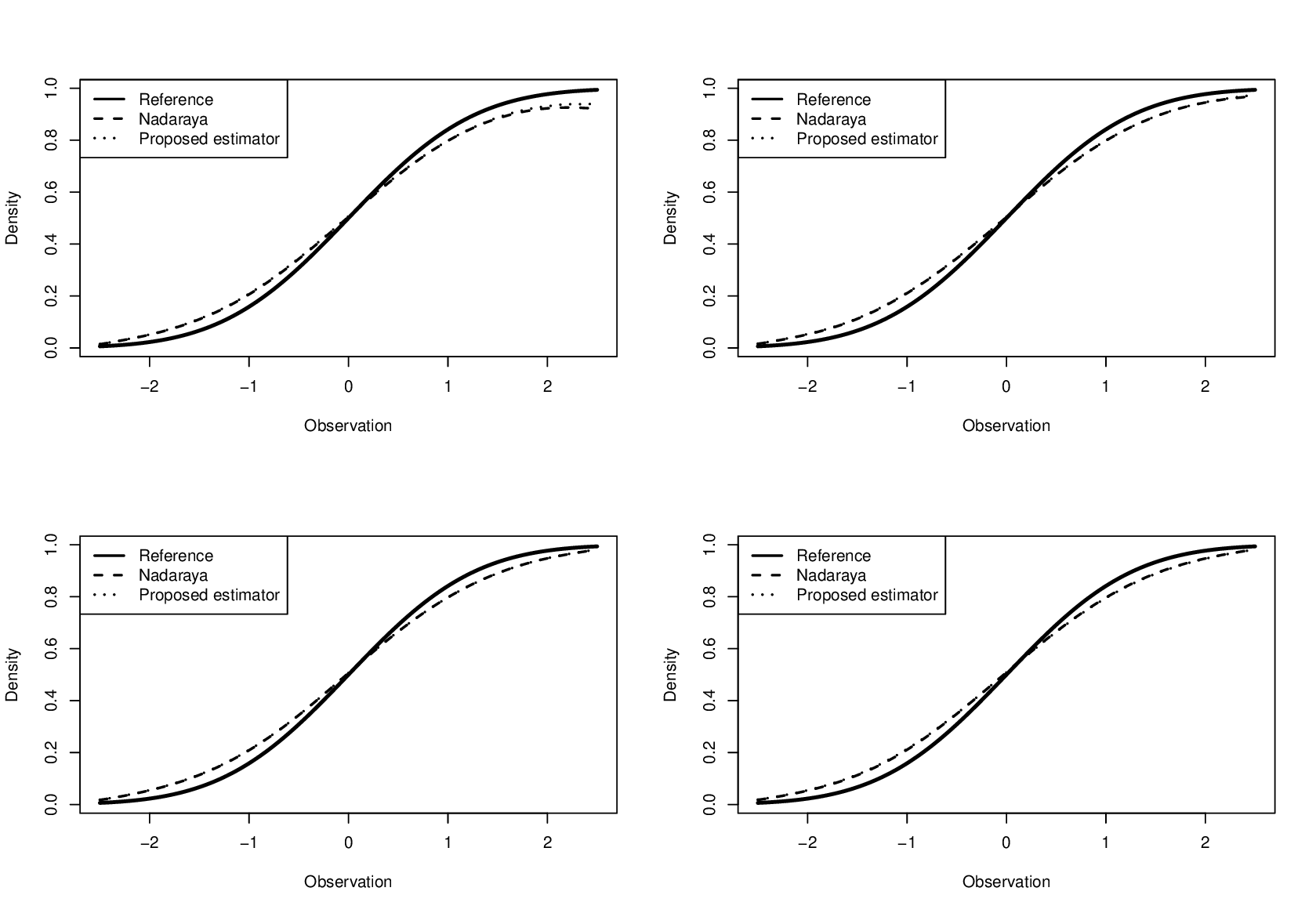}
\end{center}
\caption{Qualitative comparison between the deconvolution Nadaraya's estimator~(\ref{eq:Nada64}) and the proposed estimator~(\ref{eq:Fn1}) with the choice of the stepsize $\left(\gamma_n\right)=\left(n^{-1}\right)$, for $500$ samples of size $200$, with $\texttt{NSR}$ equal respectively to $5\%$ (in the top left panel), equal to $10\%$ (in the top right panel), equal to $20\%$ (in the down left panel) and equal to $30\%$ (in the down right panel) for the standard normal distribution $X\sim \mathcal{N}\left(0,1\right)$.}
\label{Fig:2}
\end{figure}

\begin{table}[h]
\begin{eqnarray*}
\begin{tabular}{lccccccc}
& Nadaraya  & estimator $1$  &estimator $2$  & estimator $3$ & estimator $4$\\ \hline
$n=25$&&  &$\texttt{NSR}=5\%$  &&& \\
$RMRE$ &  $0.0948$ & $0.0982$ & $0.0903$ & $0.0915$ & $0.0940$\\
$Cor$  &  $0.995$ & $ 0.995$ & $ 0.995$ & $ 0.995$ & $ 0.995$ \\
$\texttt{CPU}$ &$10$ & $5$ & $5$ & $5$ & $5$ \\
$n=50$&&  & &&& \\
$RMRE$ &  $0.0768$ & $0.0789$ & $0.0751$ & $0.0753$ & $0.0760$\\
$Cor$  &  $0.997$ & $ 0.997$ & $ 0.997$ & $ 0.997$ & $ 0.997$ \\
$\texttt{CPU}$ &$43$ & $25$ & $24$ & $23$ & $24$ \\
$n=150$&&  &  &&&\\
$RMRE$ &  $0.0357$ & $0.0373$ & $0.0344$ & $0.0340$ & $ 0.0341$\\
$Cor$  &  $0.998$ & $ 0.998$ & $ 0.998$ & $ 0.998$ & $ 0.998$\\
$\texttt{CPU}$ &$403$ & $215$ & $213$ & $212$ & $216$ \\ 
\hline
$n=25$&&  &$\texttt{NSR}=10\%$  &&& \\
$RMRE$ &  $0.0946$ & $0.1030$ & $0.0927$ & $0.0916$ & $0.0931$\\
$Cor$  &  $0.994$ & $ 0.994$ & $ 0.994$ & $ 0.994$ & $ 0.994$ \\
$\texttt{CPU}$ &$11$ & $6$ & $6$ & $6$ & $6$ \\
$n=50$&&  &  &&& \\
$RMRE$ &  $0.0803$ & $0.0808$ & $0.0749$ & $0.0748$ & $0.0755$\\
$Cor$  &  $0.997$ & $ 0.997$ & $ 0.997$ & $ 0.997$ & $ 0.997$ \\
$\texttt{CPU}$ &$42$ & $23$ & $23$ & $24$ & $25$ \\
$n=150$&&  &  &&&\\
$RMRE$ &  $0.0497$ & $0.0490$ & $0.0467$ & $0.0469$ & $0.0469$\\
$Cor$  &  $0.998$ & $ 0.998$ & $ 0.998$ & $ 0.998$ & $ 0.998$ \\
$\texttt{CPU}$ &$401$ & $206$ & $205$ & $208$ & $205$ \\ 
\hline
$n=25$&&  &$\texttt{NSR}=20\%$  &&& \\
$RMRE$ &  $0.0993$ & $0.1049$ & $0.0940$ & $0.0921$ & $0.0932$\\
$Cor$  &  $0.993$ & $ 0.993$ & $ 0.993$ & $ 0.993$ & $ 0.993$ \\
$\texttt{CPU}$ &$10$ & $5$ & $5$ & $5$ & $5$ \\
$n=50$&&  &  &&& \\
$RMRE$ &  $0.0812$ & $0.0837$ & $0.0805$ & $0.0805$ & $0.0809$\\
$Cor$  &  $0.997$ & $ 0.996$ & $ 0.996$ & $ 0.996$ & $ 0.996$ \\
$\texttt{CPU}$ &$43$ & $25$ & $24$ & $23$ & $23$ \\
$n=150$&&  &  &&&\\
$RMRE$ &  $0.0762$ & $0.0709$ & $0.0685$ & $0.0675$ & $0.0666$\\
$Cor$  &  $0.998$ & $ 0.998$ & $ 0.998$ & $ 0.998$ & $ 0.998$\\
$\texttt{CPU}$ &$394$ & $202$ & $204$ & $203$ & $201$ \\  
\hline
\end{tabular}
\end{eqnarray*}
\caption{Quantitative comparison between the deconvolution Nadaraya's estimator~(\ref{eq:Nada64}) and four estimators; estimator $1$ correspond to the estimator~(\ref{eq:Fn1}) with the choice of $\left(\gamma_n\right)=\left(\left[2/3\right]n^{-1}\right)$, estimator $2$ correspond to the estimator~(\ref{eq:Fn1}) with the choice of $\left(\gamma_n\right)=\left(n^{-1}\right)$, estimator $3$ correspond to the estimator~(\ref{eq:Fn1}) with the choice of $\left(\gamma_n\right)=\left(\left[4/3\right]n^{-1}\right)$ and estimator $4$ correspond to the estimator~(\ref{eq:Fn1}) with the choice of $\left(\gamma_n\right)=\left(\left[5/3\right]n^{-1}\right)$. Here we consider the normal distribution $X\sim \mathcal{N}\left(0,2\right)$ with $\texttt{NSR}=5\%$ in the first block, $\texttt{NSR}=10\%$ in the second block and $\texttt{NSR}=20\%$ in the last block, we consider three sample sizes $n=25$, $n=50$ and $n=150$, the number of simulations is $500$, and we compute the robust mean relative error ($RMRE$) and the linear correlation ($Cor$), and the $\texttt{CPU}$ time in seconds.}\label{Tab:3}
\end{table}

\begin{figure}[h]
\begin{center}
\includegraphics[width=0.65\textwidth,angle=270,clip=true,trim=40 0 0 0]{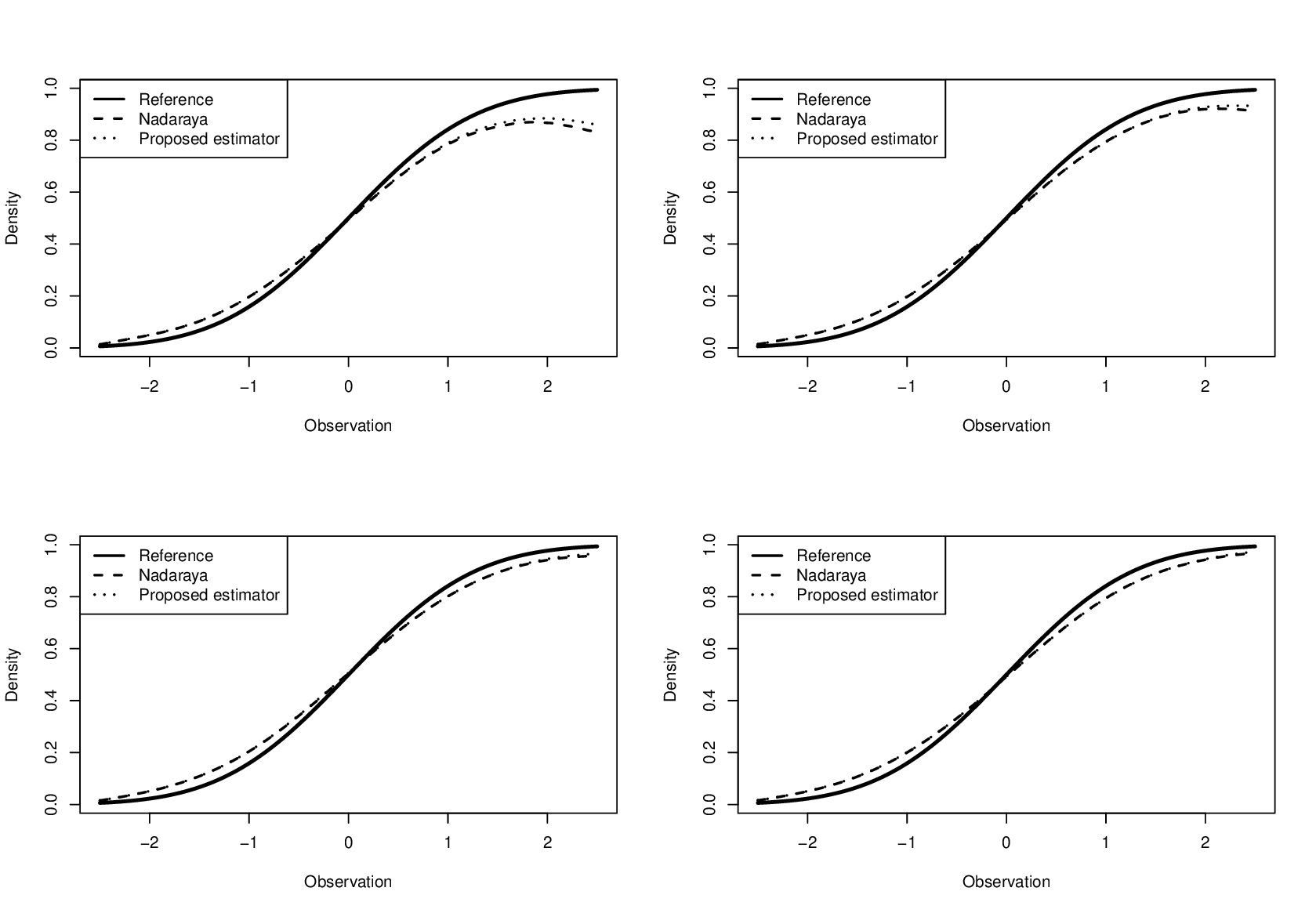}
\end{center}
\caption{Qualitative comparison between the deconvolution Nadaraya's estimator~(\ref{eq:Nada64}) and the proposed estimator~(\ref{eq:Fn1}) with the choice of the stepsize $\left(\gamma_n\right)=\left(n^{-1}\right)$, for $500$ samples of size $200$, with $\texttt{NSR}$ equal respectively to $5\%$ (in the top left panel), equal to $10\%$ (in the top right panel), equal to $20\%$ (in the down left panel) and equal to $30\%$ (in the down right panel) for the normal distribution $X\sim \mathcal{N}\left(0,2\right)$.}
\label{Fig:3}
\end{figure}

\begin{table}[h]
\begin{eqnarray*}
\begin{tabular}{lccccccc}
& Nadaraya  & estimator $1$  &estimator $2$  & estimator $3$ & estimator $4$\\ \hline
$n=25$&&  &$\texttt{NSR}=5\%$  &&& \\
$RMRE$ &  $0.0831$ & $0.0888$ & $0.0790$ & $0.0795$ & $0.0817$\\
$Cor$  &  $0.994$ & $0.994$ & $0.994$ & $0.994$ & $ 0.994$ \\
$\texttt{CPU}$ &$13$ & $7$ & $7$ & $7$ & $7$ \\
$n=50$&&  & &&& \\
$RMRE$ &  $0.0494$ & $0.0522$ & $0.0486$ & $0.0489$ & $0.0497$\\
$Cor$  &  $0.996$ & $ 0.996$ & $ 0.996$ & $ 0.996$ & $0.996$ \\
$\texttt{CPU}$ &$45$ & $25$ & $24$ & $23$ & $24$ \\
$n=150$&&  &  &&&\\
$RMRE$ &  $0.0163$ & $0.0180$ & $0.0149$ & $0.0143$ & $0.0141$\\
$Cor$  &  $0.999$ & $ 0.999$ & $ 0.999$ & $ 0.999$ & $ 0.999$\\
$\texttt{CPU}$ &$423$ & $228$ & $224$ & $225$ & $226$ \\ 
\hline
$n=25$&&  &$\texttt{NSR}=10\%$  &&& \\
$RMRE$ &  $0.0841$ & $0.0895$ & $0.0807$ & $0.0818$ & $0.0844$\\
$Cor$  &  $0.992$ & $0.992$ & $0.992$ & $0.992$ & $0.992$ \\
$\texttt{CPU}$ &$12$ & $7$ & $7$ & $7$ & $7$ \\
$n=50$&&  &  &&& \\
$RMRE$ &  $0.0547$ & $0.0579$ & $0.0540$ & $0.0539$ & $0.0544$\\
$Cor$  &  $0.994$ & $ 0.994$ & $ 0.994$ & $ 0.994$ & $ 0.994$ \\
$\texttt{CPU}$ &$44$ & $24$ & $23$ & $23$ & $24$ \\
$n=150$&&  &  &&&\\
$RMRE$ &  $0.0213$ & $0.0246$ & $0.0211$ & $0.0218$ & $ 0.0225$\\
$Cor$  &  $0.997$ & $ 0.997$ & $0.997$ & $ 0.997$ & $ 0.997$ \\
$\texttt{CPU}$ &$425$ & $226$ & $225$ & $225$ & $229$ \\ 
\hline
$n=25$&&  &$\texttt{NSR}=20\%$  &&& \\
$RMRE$ &  $0.0846$ & $0.0907$ & $0.0808$ & $0.0806$ & $0.0828$\\
$Cor$  &  $0.991$ & $ 0.991$ & $ 0.991$ & $0.991$ & $ 0.991$ \\
$\texttt{CPU}$ &$13$ & $7$ & $7$ & $7$ & $7$ \\
$n=50$&&  &  &&& \\
$RMRE$ &  $0.0582$ & $0.0616$ & $0.0580$ & $0.0581$ & $0.0587$\\
$Cor$  &  $0.993$ & $ 0.993$ & $ 0.993$ & $ 0.993$ & $0.993$ \\
$\texttt{CPU}$ &$45$ & $23$ & $24$ & $24$ & $23$ \\
$n=150$&&  &  &&&\\
$RMRE$ &  $0.0219$ & $0.0249$ & $0.0213$ & $0.0222$ & $0.0228$\\
$Cor$  &  $0.995$ & $ 0.995$ & $ 0.995$ & $ 0.995$ & $ 0.995$\\
$\texttt{CPU}$ &$435$ & $232$ & $228$ & $229$ & $230$ \\  
\hline
\end{tabular}
\end{eqnarray*}
\caption{Quantitative comparison between the deconvolution Nadaraya's estimator~(\ref{eq:Nada64}) and four estimators; estimator $1$ correspond to the estimator~(\ref{eq:Fn1}) with the choice of $\left(\gamma_n\right)=\left(\left[2/3\right]n^{-1}\right)$, estimator $2$ correspond to the estimator~(\ref{eq:Fn1}) with the choice of $\left(\gamma_n\right)=\left(n^{-1}\right)$, estimator $3$ correspond to the estimator~(\ref{eq:Fn1}) with the choice of $\left(\gamma_n\right)=\left(\left[4/3\right]n^{-1}\right)$ and estimator $4$ correspond to the estimator~(\ref{eq:Fn1}) with the choice of $\left(\gamma_n\right)=\left(\left[5/3\right]n^{-1}\right)$. Here we consider the normal mixture distribution $X\sim 1/2\mathcal{N}\left(1/2,1\right)+1/2\mathcal{N}\left(-1/2,1\right)$ with $\texttt{NSR}=5\%$ in the first block, $\texttt{NSR}=10\%$ in the second block and $\texttt{NSR}=20\%$ in the last block, we consider three sample sizes $n=25$, $n=50$ and $n=150$, the number of simulations is $500$, and we compute the robust mean relative error ($RMRE$), the linear correlation ($Cor$) and the $\texttt{CPU}$ time in seconds.}\label{Tab:4}
\end{table}

\begin{figure}[h]
\begin{center}
\includegraphics[width=0.65\textwidth,angle=270,clip=true,trim=40 0 0 0]{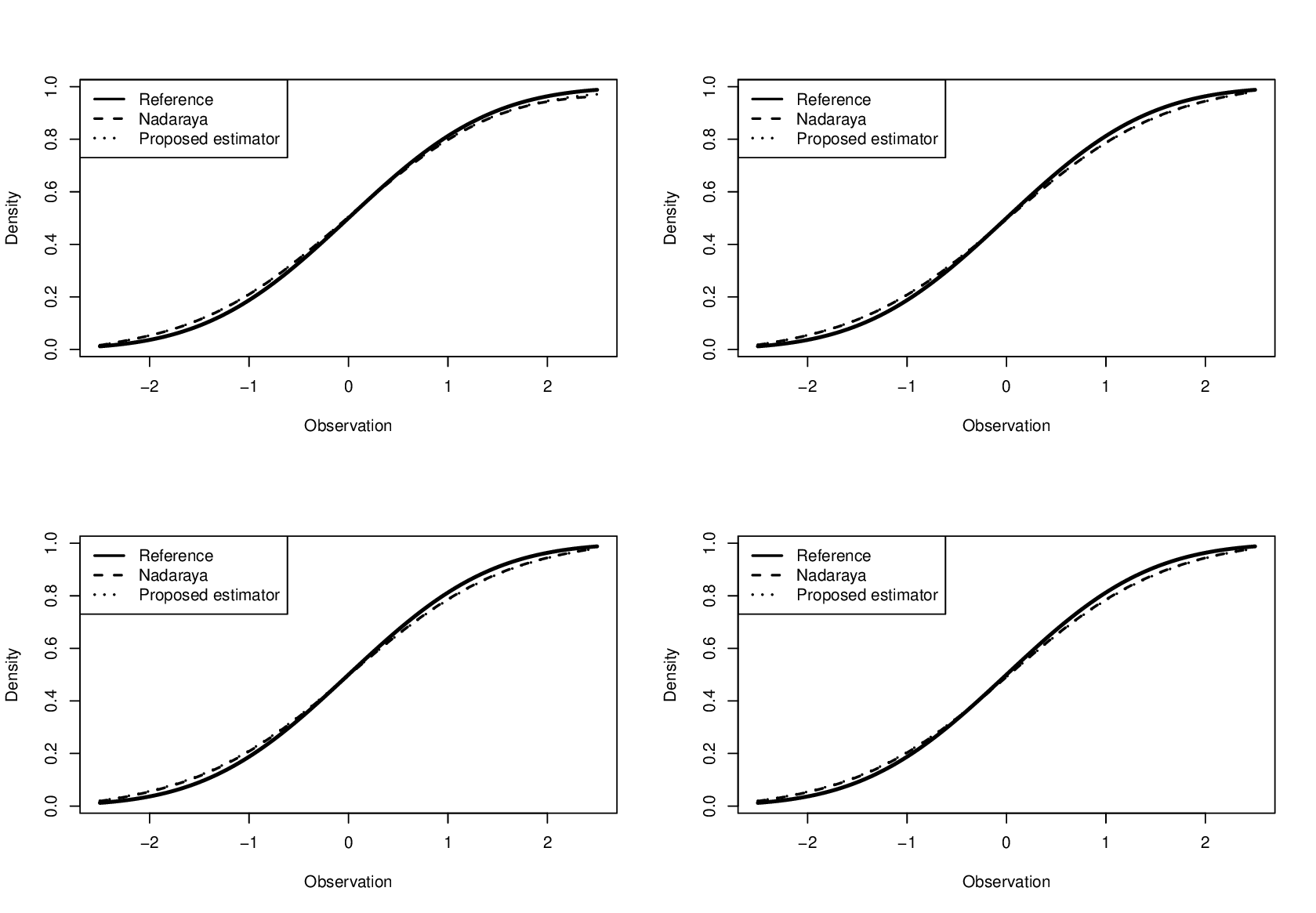}
\end{center}
\caption{Qualitative comparison between the deconvolution Nadaraya' estimator~(\ref{eq:Nada64}) and the proposed estimator~(\ref{eq:Fn1}) with the choice of the stepsize $\left(\gamma_n\right)=\left(n^{-1}\right)$, for $500$ samples of size $200$, with $\texttt{NSR}$ equal respectively to $5\%$ (in the top left panel), equal to $10\%$ (in the top right panel), equal to $20\%$ (in the down left panel) and equal to $30\%$ (in the down right panel) for the normal mixture distribution $X\sim 1/2\mathcal{N}\left(1/2,1\right)+1/2\mathcal{N}\left(-1/2,1\right)$.}
\label{Fig:4}
\end{figure}

\begin{table}[h]
\begin{eqnarray*}
\begin{tabular}{lccccccc}
& Nadaraya  & estimator $1$  &estimator $2$  & estimator $3$ & estimator $4$\\ 
\hline
$n=25$&&  &$\texttt{NSR}=5\%$  &&& \\
$RMRE$ &  $0.1298$ & $0.1336$ & $0.1239$ & $0.1242$ & $0.1265$\\
$Cor$  &  $0.955$ & $ 0.954$ & $ 0.953$ & $ 0.952$ & $ 0.952$ \\
$\texttt{CPU}$ &$9$ & $5$ & $5$ & $5$ & $5$ \\
$n=50$&&  &  &&& \\
$RMRE$ &  $0.1263$ & $0.1274$ & $0.1217$ & $0.1217$ & $0.1224$\\
$Cor$  &  $0.965$ & $ 0.964$ & $ 0.964$ & $ 0.963$ & $ 0.962$ \\
$\texttt{CPU}$ &$38$ & $20$ & $21$ & $20$ & $21$ \\
$n=150$&&  &  &&&\\
$RMRE$ &  $0.0808$ & $0.0790$ & $0.0759$ & $0.0751$ & $0.0748$\\
$Cor$  &  $0.984$ & $ 0.984$ & $ 0.985$ & $ 0.983$ & $ 0.982$\\
$\texttt{CPU}$ &$384$ & $198$ & $199$ & $197$ & $198$ \\ 
\hline
$n=25$&&  &$\texttt{NSR}=10\%$  &&& \\
$RMRE$ &  $0.1350$ & $0.1403$ & $0.1300$ & $0.1297$ & $0.1317$\\
$Cor$  &  $0.939$ & $ 0.938$ & $0.938$ & $ 0.939$ & $ 0.939$ \\
$\texttt{CPU}$ &$9$ & $5$ & $5$ & $5$ & $5$ \\
$n=50$&&  &  &&& \\
$RMRE$ &  $0.1284$ & $0.1311$ & $0.1250$ & $0.1249$ & $0.1257$\\
$Cor$  &  $0.950$ & $ 0.949$ & $ 0.949$ & $ 0.948$ & $ 0.948$ \\
$\texttt{CPU}$ &$39$ & $20$ & $19$ & $20$ & $21$ \\
$n=150$&&  &  &&&\\
$RMRE$ &  $0.1190$ & $0.1092$ & $0.1073$ & $0.1064$ & $0.1054$\\
$Cor$  &  $0.942$ & $ 0.944$ & $ 0.954$ & $ 0.954$ & $0.953$ \\
$\texttt{CPU}$ &$392$ & $204$ & $203$ & $202$ & $204$ \\ 
\hline
$n=25$&&  &$\texttt{NSR}=20\%$  &&& \\
$RMRE$ &  $0.1669$ & $0.1525$ & $0.1509$ & $0.1494$ & $0.1479$\\
$Cor$  &  $0.934$ & $ 0.934$ & $ 0.944$ & $ 0.943$ & $ 0.943$ \\
$\texttt{CPU}$ &$9$ & $5$ & $5$ & $5$ & $5$ \\
$n=50$&&  & &&& \\
$RMRE$ &  $0.1363$ & $0.1382$ & $0.1289$ & $0.1289$ & $0.1305$\\
$Cor$  &  $0.944$ & $ 0.944$ & $ 0.948$ & $ 0.949$ & $ 0.949$ \\
$\texttt{CPU}$ &$37$ & $19$ & $21$ & $21$ & $20$ \\
$n=150$&&  &  &&&\\
$RMRE$ &  $0.1258$ & $0.1213$ & $0.1160$ & $0.1151$ & $0.11508$\\
$Cor$  &  $0.933$ & $ 0.938$ & $ 0.937$ & $ 0.937$ & $ 0.938$\\ 
$\texttt{CPU}$ &$378$ & $195$ & $197$ & $194$ & $194$ \\ 
\hline
\end{tabular}
\end{eqnarray*}
\caption{Quantitative comparison between the deconvolution Nadaraya's estimator~(\ref{eq:Nada64}) and four estimators; estimator $1$ correspond to the estimator~(\ref{eq:Fn1}) with the choice of $\left(\gamma_n\right)=\left(\left[2/3\right]n^{-1}\right)$, estimator $2$ correspond to the estimator~(\ref{eq:Fn1}) with the choice of $\left(\gamma_n\right)=\left(n^{-1}\right)$, estimator $3$ correspond to the estimator~(\ref{eq:Fn1}) with the choice of $\left(\gamma_n\right)=\left(\left[4/3\right]n^{-1}\right)$ and estimator $4$ correspond to the estimator~(\ref{eq:Fn1}) with the choice of $\left(\gamma_n\right)=\left(\left[5/3\right]n^{-1}\right)$. Here we consider the exponetial  distribution $X\sim \mathcal{E}\left(1/2\right)$ with $\texttt{NSR}=5\%$ in the first block, $\texttt{NSR}=10\%$ in the second block and $\texttt{NSR}=20\%$ in the last block, we consider three sample sizes $n=25$, $n=50$ and $n=150$, the number of simulations is $500$, and we compute the robust mean relative error ($RMRE$), the linear correlation ($Cor$) and the $\texttt{CPU}$ time in seconds.}\label{Tab:5}
\end{table}

\begin{figure}[h]
\begin{center}
\includegraphics[width=0.65\textwidth,angle=270,clip=true,trim=40 0 0 0]{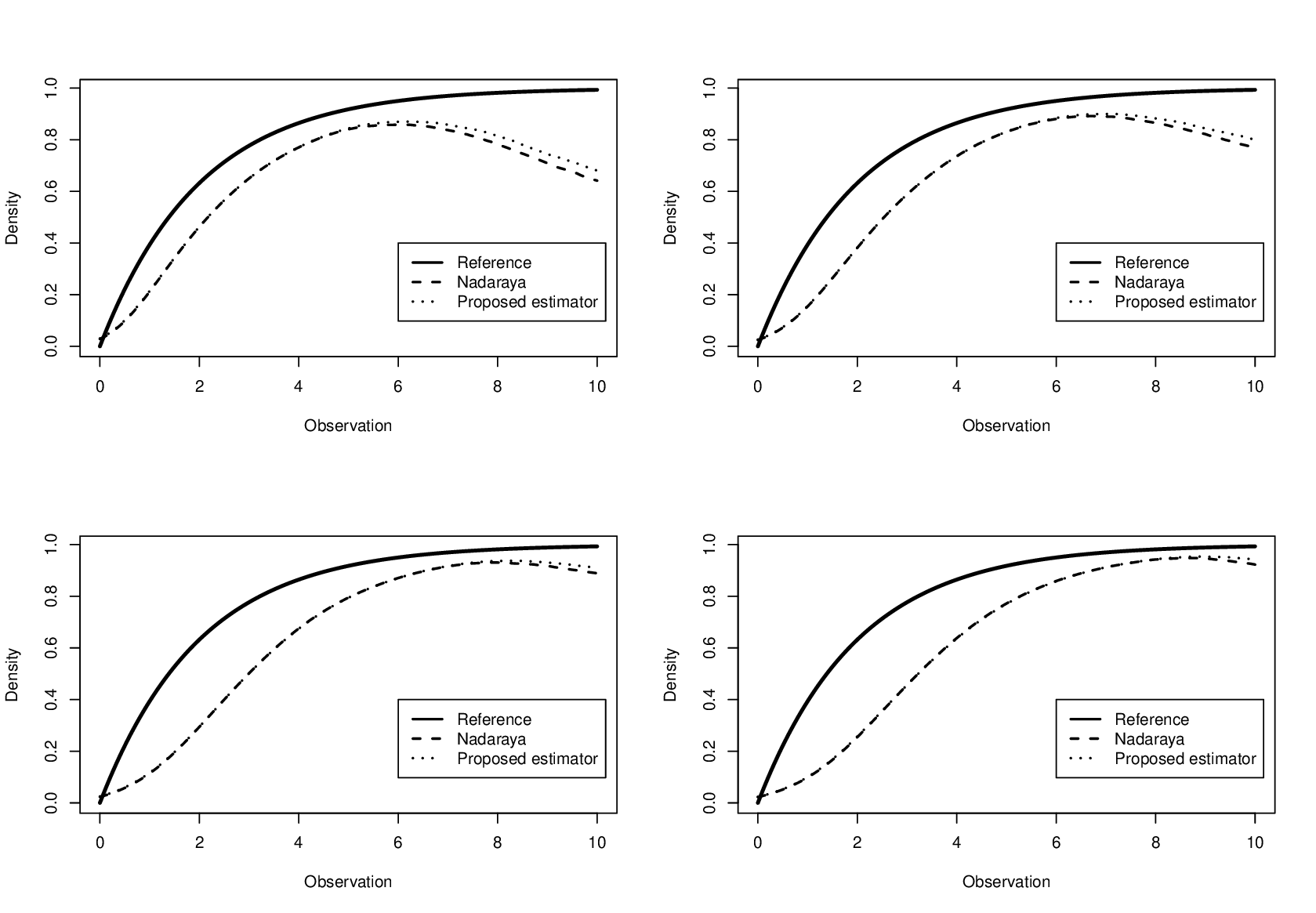}
\end{center}
\caption{Qualitative comparison between the deconvolution Nadaraya's estimator~(\ref{eq:Nada64}) and the proposed estimator~(\ref{eq:Fn1}) with the choice of the stepsize $\left(\gamma_n\right)=\left(n^{-1}\right)$, for $500$ samples of size $200$, with $\texttt{NSR}$ equal respectively to $5\%$ (in the top left panel), equal to $10\%$ (in the top right panel), equal to $20\%$ (in the down left panel) and equal to $30\%$ (in the down right panel) for the exponetial distribution $X\sim \mathcal{E}\left(1/2\right)$.}
\label{Fig:5}
\end{figure}

From tables~\ref{Tab:1},~\ref{Tab:2},~\ref{Tab:3},~\ref{Tab:4} and \ref{Tab:5}, we conclude that
\begin{enumerate}
\item[(i)] in all the cases, the $RMRE$ of the proposed distribution estimator~(\ref{eq:Fn1}), with the choice of the stepsize $\left(\gamma_n\right)=\left(n^{-1}\right)$ is smaller than the deconvolution Nadaraya's kernel distribution estimator~(\ref{eq:Nada64}).

\item[(ii)] the $RMRE$ decrease as the sample size increase.
\item[(iii)] the $RMRE$ increase as the value of $\texttt{NSR}$ increase.
\item[(iv)] the $\texttt{CPU}$ time are approximately two times faster using the proposed distribution estimator~(\ref{eq:Fn1}) compared
to the deconvolution Nadaraya's kernel distribution estimator~(\ref{eq:Nada64}).
\item[(v)] the $Cor$ increase as the sample size increase.
\item[(vi)] the $RMRE$ decrease as the value of $\texttt{NSR}$ increase.

\end{enumerate}

From figures~\ref{Fig:1},~\ref{Fig:2},~\ref{Fig:3},~\ref{Fig:4} and~\ref{Fig:5}, we conclude that, our proposed kernel distribution estimator~(\ref{eq:Fn1}), with the choice of the stepsize $\left(\gamma_n\right)=\left(n^{-1}\right)$ can be closer to the true distribution function as compared to the deconvolution Nadaraya's kernel distribution estimator~(\ref{eq:Nada64}), especially for small $\texttt{NSR}$. For our last choice of distribution function (see~\ref{Fig:5}), even when the value of $\texttt{NSR}$ is equal to $30\%$ our proposed estimator is closer to the true distribution function.

\subsection{Real Dataset}\label{subsection:real}

\paragraph{Salmon Dataset:}
This data is from~\citet{Sim96}. It concerns the size of the annual spawning stock and its production of new catchable-sized fish for 1940 through 1967 for the Skeena river sockeye salmon stock (in thousands of fish).\\
The dataset was available in the R package \texttt{idr} and contained $28$ observations on the following three variables; year, spawness and recruits, for more details see~\citet{Sim96}.\\
In order to investigate the comparison between the two estimators, we consider the annual recruits : for $500$ samples of Laplacian errors $\varepsilon\sim \mathcal{E}d\left(\sigma\right)$, with $\texttt{NSR}\in [5\%,30\%]$. For each fixed $\texttt{NSR}$, we computed the mean (over the $500$ samples) of $I_1$, $I_2$, $h_n$ and $AMISE^*$. The plug-in estimators~(\ref{plugin:hoptima:recursive}),~(\ref{plugin:hoptim:rose}) requires two kernels to estimate $I_1$ and $I_2$. In both cases we use the normal kernel with $b_n$ and $b_n^{\prime}$ are given in (\ref{eq:h:initial}), with $\beta$ equal respectively to $2/9$ and $1/6$.  

\begin{table}[h]
\begin{eqnarray*}
\begin{tabular}{lccccccc}
& $I_1$ & $I_2$ & $h_n$ & $AMISE^*$\\
\hline
&&$\texttt{NSR}=5\%$& &\\
\hline 
Nadaraya & $1.14e^{-01}$ & $5.47e^{-04}$ & $ 0.661$ & $ 6.09e^{-04}$\\
Proposed estimator & $1.15e^{-01}$ & $ 1.24e^{-02}$ & $ 0.368$ & $2.49e^{-04}$\\
\hline
&&$\texttt{NSR}=10\%$& &\\
\hline 
Nadaraya & $1.11e^{-01} $ & $4.84e^{-04} $ & $0.825$ & $5.66e^{-04}$\\
Proposed estimator & $1.12e^{-01}$ & $4.05e^{-04}$ & $ 0.819 $ & $1.52e^{-04}$\\
\hline
&&$\texttt{NSR}=20\%$& &\\ 
\hline 
Nadaraya & $1.07e^{-01}$ & $4.31e^{-04}$ & $ 1.025$ & $ 5.17e^{-04}$\\
Proposed estimator & $1.08e^{-01}$ & $3.67e^{-04}$ & $1.020$ & $3.13e^{-04}$\\
\hline
&&$\texttt{NSR}=30\%$& & \\ 
\hline 
Nadaraya & $1.03e^{-01}$ & $ 4.16e^{-04}$ & $ 1.167$ & $ 4.95e^{-04}$\\
Proposed estimator & $1.05e^{-01}$ & $ 3.83e^{-04}$ & $ 1.150$ & $ 4.86e^{-04}$\\
\hline
\end{tabular}
\end{eqnarray*}
\caption{The comparison between the $AMISE^*$ of the deconvolution Nadaraya's distribution estimator~(\ref{eq:Nada64}) and the $AMISE^*$ of the proposed distribution estimator~(\ref{eq:Fn1}) with the choice of the stepsize $\left(\gamma_n\right)=\left(n^{-1}\right)$ via the \texttt{Salvister} data of the package \texttt{kerdiest} and through a plug-in method, with $\texttt{NSR}$ equal to $5\%$ in the first block, $10\%$ in the second block, $20\%$ in the third block  and $30\%$ in the last block the number of simulations is $500$.}\label{Tab:6}
\end{table}

\begin{figure}[h]
\begin{center}
\includegraphics[width=0.65\textwidth,angle=270,clip=true,trim=40 0 0 0]{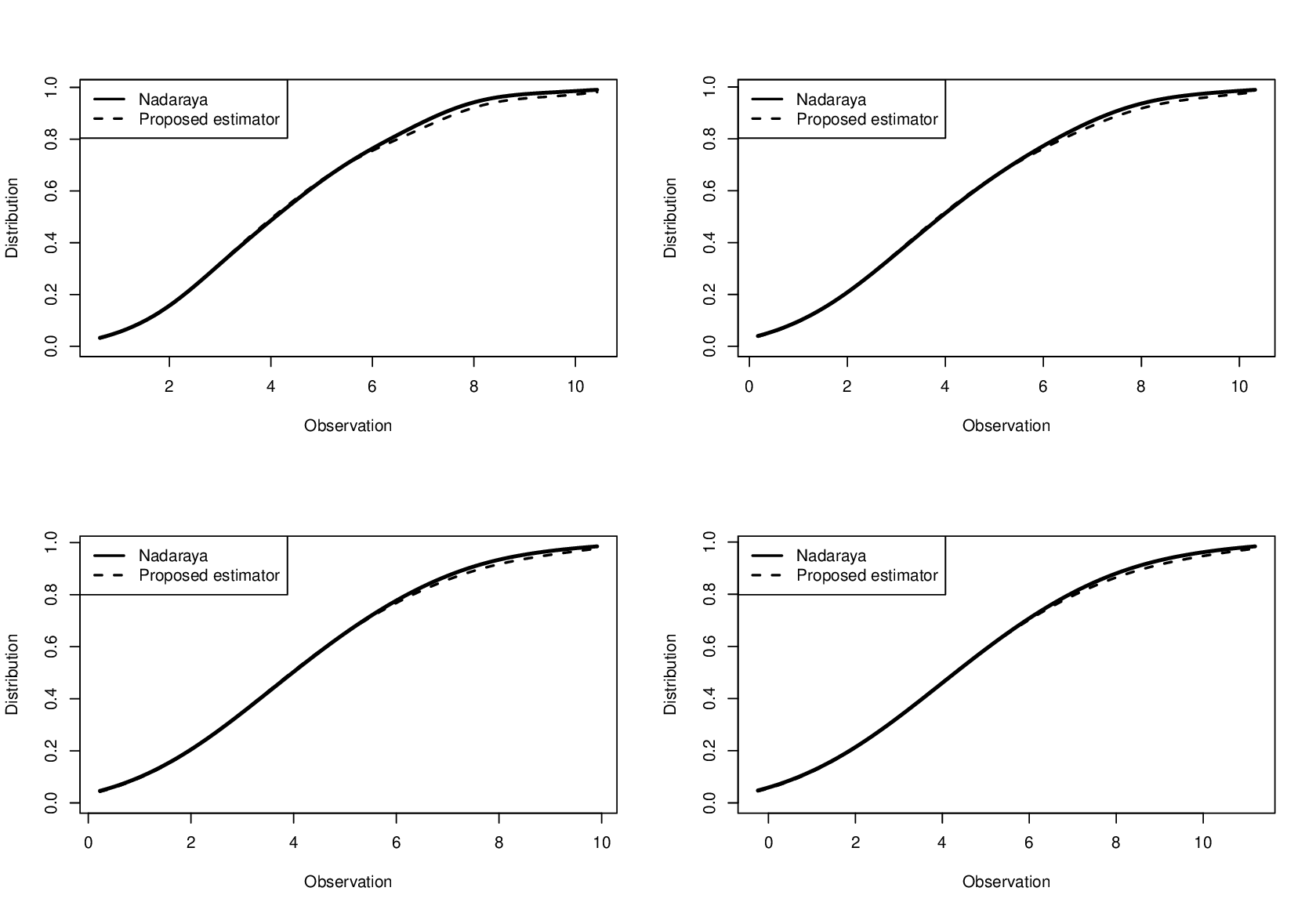}
\end{center}
\caption{Qualitative comparison between the deconvolution Nadaraya's kernel estimator~(\ref{eq:Nada64}) and the proposed estimator~(\ref{eq:Fn1}) with the choice of the stepsize $\left(\gamma_n\right)=\left(n^{-1}\right)$, for $500$ samples of Laplacian errors with $\texttt{NSR}$ equal respectively to $5\%$ (in the top left panel), equal to $10\%$ (in the top right panel), equal to $20\%$ (in the down left panel) and equal to $30\%$ (in the down right panel) for the \texttt{salmon} data of the package \texttt{idr} and through a plug-in method.}
\label{Fig:6}
\end{figure}

From the table~\ref{Tab:6}, we conclude that, the $\widehat{AMISE}^*$ of proposed estimator is quite better than the $\widetilde{AMISE}^*$ of the deconvolution Nadaraya's kernel distribution estimator. 
From the figure~\ref{Fig:6}, we conclude that the two estimators present a quite similar behavior for all the fixed $\texttt{NSR}$.
\section{Conclusion}\label{section:conclusion}
This paper propose an automatic selection of the bandwidth of a distribution function in the case of deconvolution kernel estimators with Laplace measurement errors. The estimators are compared to the deconvolution distribution estimator~(\ref{eq:Nada64}). We showed that using the selected bandwidth and the stepsizes $\left(\gamma_n\right)=\left(n^{-1}\right)$, the proposed estimator will be better than the estimator~(\ref{eq:Nada64}) for small sample setting and when the error variance is controlled by the noise to signal ratio. The simulation study corroborated these theoretical results. Moreover, the simulation results indicate that the proposed estimator was more computing efficiency than the estimator~(\ref{eq:Nada64}).\\

In conclusion, the proposed estimators allowed us to obtain quite better  results then the deconvolution Nadaraya's estimator. Moreover, we plan to make an extensions of our method in future and to consider the case of a regression function (see~\citet{Mok09b} and~\citet{Sla15a,Sla15b,Sla15c,Sla16}) in the error-free context, and to consider the case of supersmooth measurements error distribution (e.g. normal distribution).

\section{Technical proofs} \label{section:proof}
Throughout this section we use the following notations:
\begin{eqnarray}
Z_n\left(x\right)&=&\mathcal{K}\left(\frac{x-X_n}{h_n}\right)\nonumber\\
Z_n^{\varepsilon}\left(x\right)&=&\mathcal{K}^{\varepsilon}\left(\frac{x-Y_n}{h_n}\right)\label{eq:Zn:e}\\
\mu_2\left(K\right)&=&\int_{\mathbb{R}}z^2K\left(z\right)dz\nonumber\\
\psi\left(K\right)&=&\int_{\mathbb{R}}zK\left(z\right)\mathcal{K}\left(z\right)dz\label{eq:psi}
\end{eqnarray}
Let us first state the following technical lemma. 

\begin{lemma}\label{lemma:Tech} 
Let $\left(v_n\right)\in \mathcal{GS}\left(v^*\right)$, 
$\left(\gamma_n\right)\in \mathcal{GS}\left(-\alpha\right)$, and $m>0$ 
such that $m-v^*\xi>0$ where $\xi$ is defined in~\eqref{eq:xi}. We have 
\begin{eqnarray*}
\lim_{n \to +\infty}v_n\Pi_n^{m}\sum_{k=1}^n\Pi_k^{-m}\frac{\gamma_k}{v_k}
=\frac{1}{m-v^*\xi}. 
\end{eqnarray*}
Moreover, for all positive sequence $\left(\alpha_n\right)$ such that $\lim_{n \to +\infty}\alpha_n=0$, and all $\delta \in \mathbb{R}$,
\begin{eqnarray*}
\lim_{n \to +\infty}v_n\Pi_n^{m}\left[\sum_{k=1}^n \Pi_k^{-m} \frac{\gamma_k}{v_k}\alpha_k+\delta\right]=0.
\end{eqnarray*}
\end{lemma}
Lemma \ref{lemma:Tech} is widely applied throughout the proofs. Let us underline that it is its application, which requires Assumption $(A3)(iii)$ on the limit of $(n\gamma_n)$ as $n$ goes to infinity.\\

Our proofs are organized as follows. Proposition \ref{prop:bias:var:fn} in Section \ref{Section:prop:bias:var:fn}, Theorem \ref{theo:TLC} in Section \ref{section:proof:TLC}.
\subsection{Proof of Proposition~\ref{prop:bias:var:fn}} \label{Section:prop:bias:var:fn}
\begin{proof}
In view of~\eqref{eq:Fn} and~\eqref{eq:Zn:e}, we have
\begin{eqnarray}
\lefteqn{F_{n,X}\left(x\right)-F_X\left(x\right)}\nonumber\\
&=&\left(1-\gamma_n\right)\left(F_{n-1,X}\left(x\right)-F_X\left(x\right)\right)+\gamma_n\left(Z_n^{\varepsilon}\left(x\right)-F_X\left(x\right)\right)\nonumber\\
&=&\sum_{k=1}^{n-1}\left[\prod_{j=k+1}^{n}\left(1-\gamma_j\right)\right]\gamma_k\left(Z_k^{\varepsilon}\left(x\right)-F_X\left(x\right)\right)+\gamma_n\left(Z_n^{\varepsilon}\left(x\right)-F_X\left(x\right)\right)\nonumber\\
&&+\left[\prod_{j=1}^{n}\left(1-\gamma_j\right)\right]\left(F_{0,X}\left(x\right)-F_X\left(x\right)\right)\nonumber\\
&=&\Pi_n\sum_{k=1}^n\Pi_k^{-1}\gamma_k\left(Z_k^{\varepsilon}\left(x\right)-F_X\left(x\right)\right)+\Pi_n\left(F_{0,X}\left(x\right)-F_X\left(x\right)\right).
\label{encore}
\end{eqnarray}
It follows that
\begin{eqnarray*}
\mathbb{E}\left(F_{n,X}\left(x\right)\right)-F_X\left(x\right)&=&\Pi_n\sum_{k=1}^{n}\Pi_k^{-1}\gamma_k\left(\mathbb{E}\left(Z_k^{\varepsilon}\left(x\right)\right)-F_X\left(x\right)\right)+\Pi_n\left(F_{0,X}\left(x\right)-F_X\left(x\right)\right).
\end{eqnarray*}
Moreover, an interchange of expectation and integration, justified by Fubini's Theorem and assumptions $\left(A1\right)$ and $\left(A2\right)$, shows that
\begin{eqnarray*}
\mathbb{E}\left\{Z_k^{\varepsilon}\left(x\right)\vert X_k\right\}=Z_k\left(x\right),
\end{eqnarray*}
which ensure that
\begin{eqnarray*}
\mathbb{E}\left[Z_k^{\varepsilon}\left(x\right)\right]=\mathbb{E}\left[Z_k\left(x\right)\right].
\end{eqnarray*}
Moreover, by integration by parts, we have
\begin{eqnarray}
\mathbb{E}\left[Z_k\left(x\right)\right]&=&\int_{\mathbb{R}}\mathcal{K}\left(\frac{x-y}{h_k}\right)f_X\left(y\right)dy\nonumber\\
&=&\int_{\mathbb{R}}K\left(z\right)F_X\left(x+zh_k\right)dz.\label{eq:EZk}
\end{eqnarray}
It follows that
\begin{eqnarray}
\mathbb{E}\left[Z_k\left(x\right)\right]-F\left(x\right)&=&\int_{\mathbb{R}}K\left(z\right)\left[F_X\left(x+zh_k\right)-F_X\left(x\right)\right]dz\nonumber\\
&=&\frac{h_k^2}{2}f_X^{\prime}\left(x\right)\mu_2\left(K\right)+\beta_k\left(x\right),\label{eq:EZn}
\end{eqnarray}
with
\begin{eqnarray*}
\beta_k\left(x\right)=\int_{\mathbb{R}}K\left(z\right)\left[F_X\left(x+zh_k\right)-F_X\left(x\right)-zh_kf_X\left(x\right)-\frac{1}{2}z^2h_k^2f_X^{\prime}\left(x\right)\right]dz,
\end{eqnarray*}
and, since $F_X$ is bounded and continuous at $x$, we have $\lim_{k\to \infty}\beta_k\left(x\right)=0$. In the case $a\leq \alpha/7$, we have $\lim_{n\to \infty}\left(n\gamma_n\right)>2a$; the application of Lemma~\ref{lemma:Tech} then gives 
\begin{eqnarray*}
\mathbb{E}\left[F_{n,X}\left(x\right)\right]-F_X\left(x\right)&=&\frac{1}{2}f_X^{\prime}\left(x\right)\int_{\mathbb{R}}z^2K\left(z\right)dz\Pi_n\sum_{k=1}^n\Pi_k^{-1}\gamma_kh_k^2\left[1+o\left(1\right)\right]+\Pi_n\left(F_{0,X}\left(x\right)-F_X\left(x\right)\right)\nonumber\\
&=&\frac{1}{2\left(1-2a\xi\right)}f_X^{\prime}\left(x\right)\mu_2\left(K\right)\left[h_n^2+o\left(1\right)\right],\label{eq:EFn}
\end{eqnarray*}
and~(\ref{bias:Fn}) follows. In the case $a>\alpha/7$, we have $h_n^2=o\left(\sqrt{\gamma_nh_n^{-3}}\right)$, and $\lim_{n\to \infty}\left(n\gamma_n\right)>\left(\alpha-3a\right)/2$, then Lemma~\ref{lemma:Tech} ensures that
\begin{eqnarray*}
\mathbb{E}\left[F_{n,X}\left(x\right)\right]-F_X\left(x\right)&=&\Pi_n\sum_{k=1}^n\Pi_k^{-1}\gamma_ko\left(\sqrt{\gamma_kh_k}\right)+O\left(\Pi_n\right)\\
&=&o\left(\sqrt{\gamma_nh_n}\right).
\end{eqnarray*}
which gives~(\ref{bias:Fn:bis}). 
Now, we have
\begin{eqnarray}
Var\left[F_{n,X}\left(x\right)\right]&=&\Pi_n^2\sum_{k=1}^n\Pi_k^{-2}\gamma_k^2Var\left[Z_k^{\varepsilon}\left(x\right)\right]\nonumber\\
&=&\Pi_n^2\sum_{k=1}^n\Pi_k^{-2}\gamma_k^2\left(\mathbb{E}\left(\left(Z_k^{\varepsilon}\left(x\right)\right)^2\right)-\left(\mathbb{E}\left(Z_k\left(x\right)\right)\right)^2\right).\label{eq:varzk}
\end{eqnarray}
Moreover, by integration by parts, we have 
\begin{eqnarray}
\mathbb{E}\left(\left(Z_k^{\varepsilon}\left(x\right)\right)^2\right)&=&\int_{\mathbb{R}}\left(\mathcal{K}^{\varepsilon} \left(\frac{x-y}{h_k}\right)\right)^2f_Y\left(y\right)dy\nonumber\\
&=&2\int_{\mathbb{R}}K^{\varepsilon}\left(z\right)\mathcal{K}^{\varepsilon}\left(-z\right)F_Y\left(x+zh_k\right)dz\nonumber\\
&=&F_Y\left(x\right)-h_kf_Y\left(x\right)\psi\left(K^{\varepsilon}\right)+\nu_k\left(x\right),\label{eq:EZk2}
\end{eqnarray}
with 
\begin{eqnarray*}
\nu_k\left(x\right)=2\int_{\mathbb{R}}K^{\varepsilon}\left(z\right)\mathcal{K}^{\varepsilon}\left(-z\right)\left[F_Y\left(x+zh_k\right)-F_Y\left(x\right)-zh_kf_Y\left(x\right)\right]dz.
\end{eqnarray*}
Let us now state the following lemma:
\begin{lemma}\label{lem:tec1}
Let Assumptions $\left(A1\right)-\left(A2\right)$ hold, then we have
\begin{eqnarray*}
\psi\left(K^{\varepsilon}\right)&=&-\frac{1}{4\sqrt{\pi}}\left(\left(\frac{\sigma}{h_k}\right)^4+o\left(1\right)\right).
\end{eqnarray*}
\end{lemma}
\begin{proof}
First, under the assumptions $\left(A1\right)$ and $\left(A2\right)$, we have $\phi_{\varepsilon}\left(t\right)=\left(1+\sigma^2t^2\right)^{-1}$ and $\phi_K\left(t\right)=\exp\left(-t^2/2\right)$, then, it follows from~(\ref{eq:kernel:keps}), that
\begin{eqnarray*}
K^{\varepsilon}\left(u\right)&=&\frac{1}{2\pi}\int_{\mathbb{R}}\exp\left(-itu\right)\exp\left(-t^2/2\right)\left(1+t^2\frac{\sigma^2}{h_n^2}\right)dt\\
&=&\frac{1}{2\pi}\left\{\int_{\mathbb{R}}\exp\left(-\left(itu+t^2/2\right)\right)dt+\frac{\sigma^2}{h_n^{2}}\int_{\mathbb{R}}t^2\exp\left(-\left(itu+t^2/2\right)\right)dt\right\}.
\end{eqnarray*}
Moreover, it is easy to check that $\int_{\mathbb{R}}\exp\left(-\left(itu+t^2/2\right)\right)dt=\sqrt{2\pi}$ and $\int_{\mathbb{R}}t^2\exp\left(-\left(itu+t^2/2\right)\right)dt=\sqrt{2\pi}\exp\left(-u^2/2\right)\left(1-u^2\right)$, then, it follows that
\begin{eqnarray}\label{eq:keps}
K^{\varepsilon}\left(u\right)&=&\frac{1}{\sqrt{2\pi}}\exp\left(-u^2/2\right)\left(1+\frac{\sigma^2}{h_n^{2}}\left(1-u^2\right)\right).
\end{eqnarray}
Now, we let $\phi\left(u\right)=\frac{1}{\sqrt{2\pi}}\int_{-\infty}^u\exp\left(-t^2/2\right)dt$, then, we can check that
\begin{eqnarray}\label{eq:cal:keps}
\mathcal{K}^{\varepsilon}\left(u\right)&=&\phi\left(u\right)+\frac{1}{\sqrt{2\pi}}\frac{\sigma^2}{h_n^2}u\exp\left(-u^2/2\right).
\end{eqnarray}
The combinations of equations (\ref{eq:psi}), (\ref{eq:keps}) and (\ref{eq:cal:keps}) leads to
\begin{eqnarray*}
\psi\left(K^{\varepsilon}\right)&=&\frac{1}{\sqrt{2\pi}}\int_{\mathbb{R}}u\exp\left(-u^2/2\right)\phi\left(u\right)du+\frac{1}{\sqrt{2\pi}}\frac{\sigma^2}{h_n^2}\int_{\mathbb{R}}\left(u-u^3\right)\exp\left(-u^2/2\right)\phi\left(u\right)du\\
&&+\frac{1}{2\pi}\frac{\sigma^2}{h_n^2}\int_{\mathbb{R}}u^2\exp\left(-u^2\right)du+\frac{1}{2\pi}\frac{\sigma^4}{h_n^4}\int_{\mathbb{R}}\left(u^2-u^4\right)\exp\left(-u^2\right)du\\
&=&\frac{1}{2\pi}\frac{\sigma^4}{h_n^4}\int_{\mathbb{R}}\left(u^2-u^4\right)\exp\left(-u^2\right)du+o\left(\frac{\sigma^4}{h_n^4}\right).
\end{eqnarray*}
Moreover, since $\int_{\mathbb{R}}u^2\exp\left(-u^2\right)du=\sqrt{\pi}/2$ and $\int_{\mathbb{R}}u^4\exp\left(-u^2\right)du=\frac{3}{4}\sqrt{\pi}$, we conclude the proof of Lemma~\ref{lem:tec1}.

\end{proof}
Moreover, it follows from~(\ref{eq:EZk}), that
\begin{eqnarray}
\mathbb{E}\left[Z_k\left(x\right)\right]&=&F_X\left(x\right)+\int_{\mathbb{R}}K\left(z\right)\left[F_X\left(x+zh_k\right)-F_X\left(x\right)\right]dz\nonumber\\
&=&F_X\left(x\right)+\widetilde{\nu}_k\left(x\right),\label{eq:EZk1}
\end{eqnarray}
with
\begin{eqnarray*}
\widetilde{\nu}_k\left(x\right)=\int_{\mathbb{R}}K\left(z\right)\left[F_X\left(x+zh_k\right)-F_X\left(x\right)\right]dz.
\end{eqnarray*}
Then, it follows from~(\ref{eq:varzk}), (\ref{eq:EZk2}) and (\ref{eq:EZk1}), that
\begin{eqnarray}\label{eq:var:Fn:zn}
Var\left[F_{n,X}\left(x\right)\right]&=&\left(F_Y\left(x\right)-F_X^2\left(x\right)\right)\Pi_n^2\sum_{k=1}^n\Pi_k^{-2}\gamma_k^2
-f_Y\left(x\right)\Pi_n^2\sum_{k=1}^n\Pi_k^{-2}\gamma_k^2h_k\psi\left(K^{\varepsilon}\right)\nonumber\\
&&+\left(\nu_k\left(x\right)-2F\left(x\right)\widetilde{\nu}_k\left(x\right)-\widetilde{\nu}_k^2\left(x\right)\right)\Pi_n^2\sum_{k=1}^n\Pi_k^{-2}\gamma_k^2.
\end{eqnarray}
Since $F_X$ and $F_Y$ is bounded continuous, we have $\lim_{k\to \infty}\nu_k\left(x\right)=0$ and $\lim_{k\to \infty}\widetilde{\nu_k}\left(x\right)=0$. In the case $a\geq \alpha/7$, we have $\lim_{n\to \infty}\left(n\gamma_n\right)>\left(\alpha-3a\right)/2$, and the application of Lemma~\ref{lemma:Tech} gives
\begin{eqnarray*}
Var\left[F_{n,X}\left(x\right)\right]&=&\frac{\gamma_n}{2-\alpha\xi}\left(F_Y\left(x\right)-F_X^2\left(x\right)\right)
+\frac{\sigma^4}{\sqrt{\pi}}\frac{\gamma_nh_n^{-3}}{2-\left(\alpha-3a\right)\xi}f_Y\left(x\right)+o\left(\gamma_nh_n^{-3}\right),
\end{eqnarray*}
which proves~(\ref{var:Fn}). Now, in the case $a<\alpha/7$, we have $\gamma_nh_n^{-3}=o\left(h_n^4\right)$, and $\lim_{n\to \infty}\left(n\gamma_n\right)>2a$, then the application of Lemma~\ref{lemma:Tech} gives 
\begin{eqnarray*}
Var\left[F_{n,X}\left(x\right)\right]&=&\Pi_n^2\sum_{k=1}^n\Pi_k^{-2}\gamma_ko\left(h_k^4\right)\\
&=&o\left(h_n^4\right),
\end{eqnarray*}
which proves~(\ref{var:Fn:bis}).
\end{proof}
\subsection{Proof of Theorem~\ref{theo:TLC}}\label{section:proof:TLC}
\begin{proof}
Let us at first assume that, if $a\geq \alpha/7$ then
\begin{eqnarray}\label{eq:TLC}
\sqrt{\gamma_n^{-1}h_n^3}\left(F_{n,X}\left(x\right)-\mathbb{E}\left[F_{n,X}\left(x\right)\right] \right)\stackrel{\mathcal{D}}{\rightarrow} 
\mathcal{N}\left(0
,\frac{\sigma^4}{4\sqrt{\pi}\left(2-\left(\alpha-3a\right)\xi\right)}f_Y\left(x\right)\right).
\end{eqnarray}
In the case when $a>\alpha/7$, Part 1 of Theorem~\ref{theo:TLC} follows from the combination of~\eqref{bias:Fn:bis} and~\eqref{eq:TLC}. In the case when $a=\alpha/7$, Parts 1 and 2 of Theorem~\ref{theo:TLC} follow from the combination of~\eqref{bias:Fn} and~\eqref{eq:TLC}. In the case $a<\alpha/7$,~\eqref{var:Fn:bis} implies that 
\begin{eqnarray*}
h_n^{-2}\left(F_{n,X}\left(x\right)-\mathbb{E}\left(F_{n,X}\left(x\right)\right)\right)\stackrel{\mathbb{P}}{\rightarrow}0,
\end{eqnarray*}
and the application of~\eqref{bias:Fn} gives Part 2 of Theorem~\ref{theo:TLC}.\\
We now prove~\eqref{eq:TLC}.In view of~(\ref{eq:Fn}), we have
\begin{eqnarray*}
F_{n,X}\left(x\right)-\mathbb{E}\left[F_{n,X}\left(x\right)\right]
&=&\left(1-\gamma_n\right)\left(F_{n-1,X}\left(x\right)-
\mathbb{E}\left[F_{n-1,X}\left(x\right)\right]
\right)+\gamma_n\left(Z_n^{\varepsilon}\left(x\right)-\mathbb{E}\left[Z_n\left(x\right)\right]
\right)\\
&=&\Pi_n\sum_{k=1}^n\Pi_k^{-1}\gamma_k\left(Z_k^{\varepsilon}\left(x\right)-
\mathbb{E}\left[Z_k\left(x\right)\right]\right).
\end{eqnarray*}
Set 
\begin{eqnarray*}
Y_{k}\left(x\right)=\Pi_k^{-1}\gamma_k\left(Z_k^{\varepsilon}\left(x\right)-\mathbb{E}\left(Z_k\left(x\right)\right)\right).
\end{eqnarray*}
The application of Lemma~\ref{lemma:Tech} ensures that
\begin{eqnarray*} 
v_n^2&=&\sum_{k=1}^nVar\left(Y_{k}\left(x\right)\right)\nonumber\\
&=&\sum_{k=1}^n\Pi_k^{-2}\gamma_k^2Var\left(Z_k^{\varepsilon}\left(x\right)\right)\nonumber\\
&=&\sum_{k=1}^n\Pi_k^{-2}\gamma_k^2\left[\frac{\sigma^4}{4\sqrt{\pi}}h_k^{-3}f_Y\left(x\right)+o\left(1\right)\right]\nonumber\\
&=&\frac{\gamma_n}{h_n^3\Pi_n^2}\left[\frac{\sigma^4}{4\sqrt{\pi}}\frac{1}{\left(2-\left(\alpha-3a\right)\xi\right)}f_Y\left(x\right)+o\left(1\right)\right].
\end{eqnarray*}
On the other hand, we have, for all $p>0$, 
\begin{eqnarray*}
\mathbb{E}\left[\left|Z_k^{\varepsilon}\left(x\right)\right|^{2+p}\right] &=&
O\left(1\right),
\end{eqnarray*}
and, since $\lim_{n\to\infty}\left(n\gamma_n\right)>\alpha/2$, there exists $p>0$ such that $\lim_{n\to \infty}\left(n\gamma_n\right)>\frac{1+p}{2+p}\alpha$. Applying Lemma~\ref{lemma:Tech}, we get 
\begin{eqnarray*}
\sum_{k=1}^n\mathbb{E}\left[\left|Y_{k}\left(x\right)\right|^{2+p}\right]&=&O\left(\sum_{k=1}^n \Pi_k^{-2-p}\gamma_k^{2+p}\mathbb{E}\left[\left|Z_k\left(x\right)\right|^{2+p}\right]\right)\nonumber\\
&=&O\left(\sum_{k=1}^n \Pi_k^{-2-p}\gamma_k^{2+p}\right)\\
&=&O\left(\frac{\gamma_n^{1+p}}{\Pi_n^{2+p}}\right)\nonumber,
\end{eqnarray*}
and we thus obtain 
\begin{eqnarray*}
\frac{1}{v_n^{2+p}}\sum_{k=1}^n\mathbb{E}\left[\left|Y_{k}\left(x\right)\right|^{2+p}\right]& = & O\left(\gamma_n^{p/2}h_n^{3+\frac{3}{2}p}\right)=o\left(1\right).
\end{eqnarray*}
The convergence in~\eqref{eq:TLC} then follows from the application of Lyapounov's Theorem.
\end{proof}

\section*{}
\makeatletter
\renewcommand{\@biblabel}[1]{}
\makeatother


\begin{thebibliography}{99}

\bibitem[{Altman \& Leger(1995)}]{Alt95}
{Altman, N.} \& {Leger, C.} (1995).
\newblock {Bandwidth selection for kernel distribution function estimation}.
\newblock \textit{J. Statist. Plann. Inference.} {\bf 46},~195--214.


\bibitem[{Bojanic \& Seneta(1973)}]{Boj73}
{Bojanic, R.} \& {Seneta, E.} (1973).
\newblock {A unified theory of regularly varying sequences}.
\newblock \textit{Math. Z}. {\bf 134},~91--106.


\bibitem[{Galambos \& Seneta(1973)}]{Gal73}
{Galambos, J.} \& {Seneta, E.} (1973).
\newblock {Regularly varying sequences}.
\newblock \textit{Amer. Math. Soc}. {\bf 41},~110--116.

\bibitem[{Carroll et al.(1995)}]{Car95}
\newblock {Carroll, R.J.}, {Ruppert, D.} \& {Stefanski, L.} (1995). Measurement Error in Nonlinear Models. Chapman \& Hall, London.

\bibitem[{Delaigle \& Gijbels(2004)}]{Del04}
{Delaigle, A.} \& {Gijbels, I.} (2004).
\newblock {Practical Bandwidth Selection in Deconvolution Kernel Density
Estimation}.
\newblock \textit{Comput. Statist. Data Anal.} {\bf 45},~249--267.


\bibitem[{Fan(1991)}]{Fan91}
{Fan, J.} (1991) 
\newblock {On the optimal rates of convergence for nonparametric deconvolution problems}. 
\newblock \textit{Ann. Statist} {\bf 19},~1257--1272.

\bibitem[{Hall \& Maron(1987)}]{Hal87}
{Hall, P.} \& {Maron, J.~S.} (1987).
\newblock {Estimation of integrated squared density derivatives}.
\newblock \textit{Statist. Probab. Lett}. {\bf 6}, 109--115.


\bibitem[{Mokkadem \& Pelletier(2007)}]{Mok07}
{Mokkadem, A.} \& {Pelletier, M.} (2007).
\newblock {A companion for the Kiefer-Wolfowitz-Blum stochastic approximation algorithm}.
\newblock \textit{Ann. Statist.} \textbf{35},~1749--1772.


\bibitem[{Mokkadem et~al.(2009a)Mokkadem, Pelletier \& Slaoui (2009)}]{Mok09a} 
{Mokkadem, A.} {Pelletier, M.} \& {Slaoui, Y.} (2009a).
\newblock {The stochastic approximation method for the estimation of a multivariate probability density}.
\newblock \textit{J. Statist. Plann. Inference.} {\bf 139},~2459--2478.

\bibitem[{Mokkadem et~al.(2009b)Mokkadem, Pelletier and Slaoui(2009)}]{Mok09b} 
{Mokkadem, A.} {Pelletier, M.} and {Slaoui, Y.} (2009b).
\newblock {Revisiting R\'ev\'esz's stochastic approximation method for the estimation of a regression function}.
\newblock \textit{ALEA. Latin American Journal of Probability
and Mathematical Statistics}, 6, 63--114.


\bibitem[Nadaraya(1964)]{Nad64}
{Nadaraya, E. A.} (1964).
\newblock {Some New Estimates for Distribution Functions}.
\newblock \textit{Theory Probab. Appl.} {\bf 9},~497--500.


\bibitem[R\'ev\'esz(1973)]{Rev73}
{R\'ev\'esz, P.} (1973).
\newblock {Robbins-Monro procedure in a Hilbert space and its
application in the theory of learning processes I}.
\newblock \textit{Studia Sci. Math. Hung.} {\bf 8},~391--398.

\bibitem[R\'ev\'esz(1977)]{Rev77}
{R\'ev\'esz, P.} (1977).
\newblock {How to apply the method of stochastic approximation in the non-parametric estimation of a regression function}.
\newblock \textit{Math. Operationsforsch. Statist., Ser. Statistics.} {\bf 8},~119--126.

\bibitem[{Silverman(1986)}]{Sil86} 
{Silverman, B.~W.} (1986).
\newblock {Density estimation for statistics and data analysis}.
\newblock \textit{Chapman \& Hall}, London.

\bibitem[{Simonoff(1996)}]{Sim96} 
{Simonoff, J.~S.} (1996).
\newblock {Smoothing Methods in Statistics}.
\newblock \textit{New York}, Spinger-Verlag.

\bibitem[{Slaoui(2013)}]{Sla13} 
{Slaoui, Y.} (2013).
\newblock {Large and moderate deviation principles for recursive kernel density estimators defined by stochastic approximation method}.
\newblock \textit{Serdica Math. J.} {\bf 39},~53--82.

\bibitem[{Slaoui(2014a)}]{Sla14a} 
{Slaoui, Y.} (2014a).
\newblock {Bandwidth selection for recursive kernel density estimators defined by stochastic approximation method}.
\newblock \textit{J. Probab. Stat} {\bf 2014}, ID 739640, doi:10.1155/2014/739640.

\bibitem[{Slaoui(2014b)}]{Sla14b} 
{Slaoui, Y.} (2014b).
\newblock {The stochastic approximation method for the estimation of a distribution function}.
\newblock \textit{Math. Methods Statist.} {\bf 23},~306--325.

\bibitem[{Slaoui(2014c)}]{Sla14c} 
{Slaoui, Y.} (2014c).
\newblock {Large and moderate deviation principles for kernel distribution estimator}.
\newblock \textit{Int. Math Forum.} {\bf 18},~871--890.


\bibitem[{Slaoui(2015a)}]{Sla15a} 
{Slaoui, Y.} (2015a).
\newblock {Plug-In Bandwidth selector for recursive kernel regression estimators defined by stochastic approximation method.}
\newblock \textit{Stat. Neerl.} {\bf 69},~483--509.


\bibitem[{Slaoui(2015b)}]{Sla15b}
{Slaoui, Y.} (2015b).
\newblock {Large and moderate deviation principles for averaged stochastic approximation method for the estimation of a regression function}.
\newblock \textit{Serdica Math. J.} {\bf 41}, 307--328.

\bibitem[{Slaoui(2015b)}]{Sla15c}
{Slaoui, Y.} (2015c).
\newblock {Moderate deviation principles for recursive regression estimators defined by stochastic approximation method}.
\newblock \textit{Int. J. Math. Stat.} {\bf 16}, 51--60.

\bibitem[{Slaoui(2016)}]{Sla16} 
{Slaoui, Y.} (2016). 
\newblock {Optimal bandwidth selection for semi-recursive kernel regression estimators.}
\newblock \textit{Stat. Interface} {\bf 9}, 375--388.

\bibitem[{Stefanski \& Caroll(1990)}]{Ste90}
{Stefanski, L.~A.} \& {Carroll, R.~J.} (1990).
\newblock {Deconvoluting kernel density estimators}. 
\newblock \textit{Statistics}. {\bf 2}, 169--184.

\bibitem[{Tsybakov(1990)}]{Tsy90} 
{Tsybakov, A.~B.} (1990).
\newblock {Recurrent estimation of the mode of a multidimensional distribution}.
\newblock \textit{Probl. Inf. Transm.} {\bf 8},~119--126.



\end{thebibliography}
\end{document}